\documentclass[12pt,onecolumn]{IEEEtran}

\usepackage{cite}
\usepackage{amsmath,amssymb,amsfonts}
\usepackage{algorithmic}
\usepackage{graphicx}
\usepackage{textcomp}
\usepackage{xcolor}
\usepackage{amsthm}
\usepackage[ruled,vlined]{algorithm2e}
\usepackage{subfig}
\usepackage{comment}
\DeclareMathOperator{\diag}{diag}

\newtheorem{definition}{Definition}[section]
\newtheorem{theorem}{Theorem}[section]

\newtheorem{lemma}[theorem]{Lemma}
\newtheorem{assumption}{Assumption}[section]

\newtheorem{proposition}{Proposition}[section]

\def\BibTeX{{\rm B\kern-.05em{\sc i\kern-.025em b}\kern-.08em
    T\kern-.1667em\lower.7ex\hbox{E}\kern-.125emX}}

\begin{document}

\title{Bayesian Risk-averse Model Predictive Control with Consistency and Stability Guarantees}

\author{Yingke Li,  Yifan Lin, Enlu Zhou and Fumin Zhang
\thanks{Yingke Li is with the Department of Aeronautics and Astronautics, MIT, Cambridge, MA, 02139 USA
        {\tt\small yingkeli@mit.edu}}
\thanks{Yifan Lin and Enlu Zhou are with the School of Industrial and Systems Engineering, Georgia Institute of Technology, Atlanta, GA, 30332 USA
        {\tt\small ylin429,ezhou30@gatech.edu}}
\thanks{Fumin Zhang is with the Department of Electronic and Computer Engineering, Hong Kong University of Science and Technology, Kowloon, Hong Kong, China
        {\tt\small eefumin@ust.hk}}
\thanks{The research work is supported by ONR grants N00014-19-1-2556 and N00014-19-1-2266; AFOSR grants FA9550-19-1-0283 and FA9550-22-1-0244; NSF grants CNS-1828678, S\&AS-1849228, DMS-2053489 and GCR-1934836; and NOAA grant NA16NOS0120028.}
}

\maketitle

\begin{abstract}
Model Predictive Control (MPC) is a powerful framework for constrained control, but its performance and safety can be severely degraded when the prediction model is learned online and thus remains uncertain. In this work, we develop a Bayesian risk-averse MPC framework for stochastic, discrete-time, nonlinear systems that provides theoretical guarantees on the consistency of Bayesian learning and closed-loop stability. First, we study Bayesian learning under the conditionally independent state transitions induced by feedback control and establish explicit conditions for Bayesian consistency on an infinitely countable parameter space. Second, we introduce a general notion of risk-averse asymptotic stability (RAAS), defined via comparison function classes and independent of any specific coherent risk measure or convergence rate, and we derive a risk-averse Lyapunov stability theorem together with MPC-specific stability conditions. Third, building on these foundations, we design a practical Bayesian risk-averse MPC scheme that separates epistemic (parametric) and aleatoric (disturbance) uncertainty: additive disturbances are treated in a risk-neutral fashion, while parametric uncertainty is managed via dynamically shrinking ambiguity sets constructed from Bayesian credible intervals, approximated online using particle filtering. To enable real-time implementation, we propose both an optimal and a sub-optimal receding-horizon control policy, the latter obtained by warm-starting from the previous solution, and prove that asymptotic RAAS is recovered as the Bayesian estimator becomes consistent.
\end{abstract}


\section{Introduction}

Model Predictive Control (MPC) \cite{Rawlings2017ModelDesign}, as the prime methodology for constrained control, offers a significant opportunity to exploit the abundance of data in a reliable manner, particularly while taking safety constraints into account \cite{Mesbah2018StochasticControl, Koller2019Learning-BasedExploration, Bonzanini2020SafeTrees}.
Deploying a receding horizon fashion, MPC is a practical implementation of optimal control principles to design a closed-loop controller, which inherently incorporates feedback by re-optimizing at each time step based on new measurements.

The MPC scheme relies on a sufficiently descriptive model of the system to optimize performance and ensure constraint satisfaction, rendering modeling critical for the success of the resulting control system.
The classic controller paradigms for MPC follow a strict separation of a design phase, which is carried out offline by a control engineer, and an application phase of closed-loop control, during which the formulation of the controller remains essentially unchanged. 
Recent successes in machine learning have significantly improved offline controller design and opened new avenues for learning-based control \cite{Rosolia2018LearningFramework, Hewing2020Learning-BasedControl}. 
Numerous explicit model-based approaches, such as Bayesian MPC \cite{Wabersich2020BayesianSampling, Wabersich2020PerformanceGuarantees}, have been proposed to infer the prediction model from recorded data directly.
Another revenue of methods, which rely on an implicit model description based on behavioral systems theory, has also emerged to empower Data-enabled Predictive Control (DeePC) \cite{Coulson2019Data-enabledDeepc, Berberich2021Data-DrivenGuarantees}.

The increasing availability of sensing and communication capabilities, coupled with enhanced computational power, has sparked renewed interest in automating controller design and adaptation based on data collected during operation, e.g., for improved performance, facilitated deployment, and a reduced need for manual controller tuning \cite{Thangavel2018DualApproach, Thangavel2018RobustEstimation, Lorenzen2019RobustUpdate, Kohler2020ASystems}. 
For online learning-based approaches, a primary concern is the consistency of estimation. Specifically, this refers to whether the time-varying model characterized by the online estimator can converge to the true model given sufficient time and data. 
The consistency of the Bayesian estimators with independent and identically distributed (i.i.d.) data has been extensively studied \cite{Diaconis1986OnEstimates}. 
However, when Bayesian estimators are applied to online learning-based control, they must handle correlated but conditionally independent data, as the likelihood of system transitions depends on the system state and control action.
In this paper, we fill this research gap by providing explicit conditions that guarantee Bayesian consistency in an infinitely countable parameter space under conditionally independent state transitions, where the transition kernel is conditionally independent with respect to (w.r.t.) the state and action.

An inherent limitation of online learning-based control is the unavoidable presence of uncertainty, as accurate system models can only be partially and incrementally inferred from sequentially available data. This model ambiguity is especially pronounced during the early stages of learning when data is limited. As a result, the performance and constraint satisfaction of MPC can be significantly compromised if this uncertainty is not effectively addressed \cite{Rawlings2017ModelDesign}. 

Two mainstream approaches exist to address the uncertainty in MPC: the \textit{robust} and the \textit{stochastic} approaches. 
Robust MPC strategies, in which modeling errors or disturbances are modeled as unknown but bounded quantities, provide a conservative satisfaction of hard constraints of states and control inputs for all possible uncertainty realizations \cite{Lorenzen2019RobustUpdate, Kohler2020ASystems}.
Stochastic MPC methods, on the other hand, address more realistic scenarios where modeling errors or disturbances are generally unbounded, which allow the toleration of constraint violations with a prespecified probability \cite{Mesbah2016StochasticResearch, Mesbah2018StochasticControl}.

Leveraging the probabilistic nature of uncertainty, a \textit{risk-averse} (or \textit{distributionally robust}) paradigm has been proposed by interpolating between optimistically assuming the expected case and pessimistically fixating on the worst-case. 
This risk-averse formulation effectively accounts for the impact of \textit{low-probability, high-impact} extreme events, providing significant advantages in managing uncertainty for safety-critical autonomous systems. We refer the readers to \cite{Wang2022Risk-averseControl, Coulson2021DistributionallyControl} for a broader perspective of risk-aware control theory.
A pivotal aspect when evaluating the performance of an autonomous system is its stability, specifically whether the system can reliably converge to the desired behavior. However, theoretical guarantees of risk-aware stability are often limited to certain special conditions, such as focusing on a specific coherent risk measure (e.g., CVaR in \cite{Kishida2024Risk-AwareInvariance}, polytopic risk measures in \cite{Singh2019AAlgorithms}) and/or a specific convergence rate type like exponential stability \cite{Chapman2022Risk-AwareSystems, Singh2019AAlgorithms, Sopasakis2019Risk-averseControl, Sopasakis2019Risk-averseControlb}.
This paper introduces a general notion of \textit{Risk-averse Asymptotic Stability (RAAS)} defined in terms of comparison function classes, without the restrictions of any specific coherent risk measure or convergence rate. We develop a \textit{risk-averse Lyapunov stability theorem} that explicitly identifies the conditions for RAAS. The reliance on the Lyapunov function makes these conditions applicable to general non-linear systems, in contrast to most existing works focusing on linear systems \cite{Kishida2024Risk-AwareInvariance, Chapman2022Risk-AwareSystems, Singh2019AAlgorithms}. When a system is controlled by an MPC controller, sufficient conditions for ensuring RAAS are derived in terms of the stage cost and terminal cost of the MPC controller's value function.

As an attempt to handle the uncertainty that exists in online learning-based MPC from a risk-averse perspective, we propose a Bayesian risk-averse MPC framework for stochastic, discrete, and nonlinear systems. 
The greatest challenge that impedes effective control of nonlinear stochastic systems is the entanglement between imperfect system modeling and additive stochastic disturbance.
This paper distinguishes between two sources of uncertainty: one arising from flawed estimation of unknown system parameters (usually referred to as epistemic uncertainty, or parametric uncertainty) and the other from noisy disturbances (usually referred to as intrinsic uncertainty, or aleatoric uncertainty). Since these two types of uncertainty can have markedly different impacts on the evolution of nonlinear systems, they are addressed independently. Additive random disturbances are typically modeled as i.i.d. zero-mean white noise, with limited magnitude, whose effects can be mitigated in a risk-neutral manner by taking expectations over the noise distribution. However, errors due to parametric uncertainty may undergo distorted propagation and be vulnerable to extreme events due to the nonlinearity of systems. To address the unique characteristics of parametric uncertainty, we consider applying a risk-averse perspective over the unknown parameter.
Unlike most online data-driven distributionally robust MPC approaches that design time-varying ambiguity sets for additive random disturbances \cite{Ning2021OnlineSystems, Coppens2022Data-DrivenSystems}, our approach shifts the focus to managing risks in parametric uncertainty, where the ambiguity set adapts dynamically to the sequentially updated Bayesian posteriors of the unknown parameters. 
Similar distributionally robust formulations for parametric uncertainty have been proposed specifically for Markov jump systems with unknown switching probabilities \cite{Schuurmans2020Learning-BasedFeasibility, Schuurmans2023ASystems}. Alternatively, our framework considers more general nonlinear stochastic systems with well-defined state transition kernels.

In this paper, we develop a practical Bayesian risk-averse MPC algorithm to capacitate computationally tractable implementation with theoretical guarantees of consistency and stability. 
The primary challenges for practical implementation stem from two components: calculating the Bayesian posterior update and optimizing the multi-stage value function.

We address the first challenge by utilizing a well-accepted sequential sampling-based approach, \textit{particle filter} \cite{Doucet2010ALater}, to approximate the sequential Bayesian update. 
\textit{Credible interval}, also known as Bayesian confidence interval, is a standard statistical tool for quantifying uncertainty with Bayesian posteriors.
Therefore we devise a pragmatic approach to construct a time-varying ambiguity set based on the credible interval, whose radius will shrink to zero when the Bayesian estimator is consistent.
This approach can be viewed as a generalization of classical Gaussian process-based methods\cite{Guzman2021HeteroscedasticControl, Guzman2022BayesianUncertainty}, providing greater flexibility by constructing ambiguity sets from general probability distributions, rather than being limited to the symmetric and unimodal normal distributions typical of Gaussian processes. 
Additionally, Gaussian process often suffers from increasingly high computational cost (due to the inversion of the covariance matrix) as more data points have been sampled. 
In contrast, our method maintains a constant computational cost over time by updating the Bayesian posterior on a fixed parameter space.

To deal with the challenge of multi-stage stochastic optimization, we resort to the parameterization of control policy to reduce the computational burden.
The parameterized optimal control policy is obtained by solving the corresponding time-varying risk-averse MPC problem, whose closed-loop stability is guaranteed asymptotically when the Bayesian estimator is consistent.
To further reduce the computational burden, we derive an alternative sub-optimal control policy that reduces the demands on online optimization algorithms while maintaining asymptotic stability. This sub-optimal policy is obtained by directly optimizing over a feasible warm-start control policy. Leveraging the receding horizon framework of MPC and the forward-shrinking property of ambiguity sets, the control policy from the previous step naturally provides a feasible warm-start for the next step, significantly alleviating the online optimization load.

The contributions of this paper are summarized as follows:
\begin{itemize}
    \item[(i)] \textit{Bayesian learning and consistency}: We provide explicit conditions that guarantee Bayesian consistency in an \textit{infinitely countable} parameter space under conditionally independent state transitions.
    \item[(ii)] \textit{Risk-averse MPC and stability}: We introduce a general notion of \textit{risk-averse asymptotic stability} defined in terms of comparison function classes, and rigorously develop the risk-averse Lyapunov stability theorem and risk-averse MPC stability conditions.
    \item[(iii)] \textit{Bayesian risk-averse MPC}: We derive an online Bayesian learning-based risk-averse MPC framework for general nonlinear systems. A practical algorithm is developed to capacitate computationally tractable implementation with theoretical guarantees of consistency and stability. 
\end{itemize}

This paper significantly extends our previous work \cite{Li2022Risk-AwareLearning} through the following contributions:
1) Introduce a concept of \textit{observational distinguishability} to offer an intuitive interpretation of the consistency conditions and insights to design control policies that explicitly balance exploration and exploitation;
2) Generalize the notion of \textit{asymptotic stability} from a \textit{risk-averse} perspective and establish the corresponding Lyapunov and MPC-specific stability conditions; 
3) Reformulate the risk-averse MPC problem into a nested structure solvable via dynamic programming (DP); 
4) Propose a practical Bayesian risk-averse MPC algorithm with computationally efficient implementation through sub-optimal control policies; and 
5) Provide theoretical guarantees on the consistency and stability of both optimal and sub-optimal policies.
None of these results appeared in the earlier conference version.

The remainder of this paper is organized as follows. 
Section \ref{sec:background} reviews some preliminaries on risk measures. 
In Section \ref{sec:problem formulation}, the risk-averse MPC problem is formulated with decoupled parametric uncertainty and noisy disturbances.
Section \ref{sec:Bayesian Learning} presents the sequential Bayesian learning algorithm and provides proof of its consistency. 
In Section \ref{sec:risk-averse MPC}, a general notion of \textit{risk-averse asymptotic stability} is formally defined, and the proofs of the risk-averse Lyapunov stability theorem and risk-averse MPC stability conditions are derived.
The practical Bayesian risk-averse MPC algorithm is developed in Section \ref{sec:Bayesian risk-averse MPC}, where its theoretical guarantee of closed-loop stability is also provided.
Numerical simulations of a real-world scenario are presented in Section \ref{sec:simulation} to validate the effectiveness of the proposed approach, followed by concluding remarks in Section \ref{sec:conclusion}.
All proofs of theorems have been deferred to the appendices for readability.

\section{Background}\label{sec:background}

\subsection{Risk Measure}

A sample space $(\Omega, \mathcal{F})$ consists of an abstract set $\Omega$ and $\sigma$-algebra $\mathcal{F}$ of subsets of $\Omega$. Let $\mathcal{Q}$ be a probability measure on $(\Omega, \mathcal{F})$.
Consider a linear space $\mathcal{Z}:=\mathcal{L}_{p}(\Omega, \mathcal{F}, \mathcal{Q}),p\in [1,+\infty)$, which consists of all $\mathcal{F}$-measurable functions $\phi:\Omega \rightarrow \mathbb{R}$ such that $\int_{\Omega}|\phi(\omega)|^{p} d \mathcal{Q}(\omega) \le + \infty$. 

A \textit{risk measure} is a mapping $\mathcal{R}: \mathcal{Z} \rightarrow \overline{\mathbb{R}}$, where $\overline{\mathbb{R}}:=\mathbb{R}\cup\{+\infty\}\cup\{-\infty\}$ is the extended real line.
It is called \textit{coherent} if it satisfies the following properties:
\begin{itemize}
    \item[(i)] \textit{Subadditivity}, if $Z, Z' \in \mathcal{Z}$, then $\mathcal{R}(Z+Z') \le \mathcal{R}(Z)+\mathcal{R}(Z')$. 
    \item[(ii)] \textit{Monotonicity}, if $Z, Z' \in \mathcal{Z}$, and $Z \ge Z'$, then $\mathcal{R}(Z) \ge \mathcal{R}(Z')$. 
    \item[(iii)] \textit{Translation equivariance}, if $Z \in \mathcal{Z}$, and $a \in \mathbb{R}$, then $\mathcal{R}(Z+a) = \mathcal{R}(Z)+a$. 
    \item[(iv)] \textit{Positive homogeneity}, if $Z \in \mathcal{Z}$, and $\alpha \ge 0$, then $\mathcal{R}(\alpha Z) = \alpha \mathcal{R}(Z)$.
\end{itemize}

Specifically, for any finite (real-valued) coherent risk measure, it has a dual representation of the form (cf. Theorem 6.6 in \cite{Shapiro2009LecturesProgramming}):
\begin{equation}\label{eqn:dual risk}
    \mathcal{R}(Z):= \sup_{Q \in \mathcal{A}} \mathbb{E}_{Q}(Z), \forall Z \in \mathcal{Z},
\end{equation}
where is $\mathcal{A}$ a convex bounded and weakly closed set of probability measures (distributions) on the sample space $(\Omega, \mathcal{F})$, which is often called the \textit{ambiguity set}.

\section{Problem Formulation}\label{sec:problem formulation}

Consider a stochastic, discrete-time, and possibly nonlinear dynamic system:
\begin{equation*}
\begin{split}
    &x_{k+1}=f(x_{k},u_{k},w_{k},\theta), \\
\end{split}
\end{equation*}
where $x_{k} \in \mathcal{X}$ is the system state, and $u_{k}\in \mathcal{U}$ is the applied input at time $k$. The dynamics is subject to various sources of uncertainty, which are distinguished into two categories: $\theta \in \Theta$ is a random vector describing the parametric uncertainty of the system, whose true value $\theta^{\ast}$ is unknown and need to be estimated and updated over time, and $w_{k} \in \mathcal{W}$ describes a sequence of random vectors corresponding to disturbances or process noises in the system, which are often assumed to be i.i.d.  
The state space $\mathcal{X}$, control constraint $\mathcal{U}$, parameter space $\Theta$, and disturbance space $W$ are all non-empty Borel spaces.

MPC is a form of online control in which the control policy is obtained by solving a finite-horizon optimal control problem at each sampling time instant, while only the first step control policy is employed in the system. 
This procedure is repeated at the next time instant with new measurement information incorporated.
For simplicity, the formulation of MPC is presented for the case of full state-feedback control, in which the system states are known at each sampling time instant. 
Let $N \in \mathbb{N}$ be the prediction horizon, and assume that the control horizon is equal to the prediction horizon. Define an \textit{N-stage feedback control policy} as 
$
    \boldsymbol{\mu}:=\{ \mu_{0}(\cdot), \mu_{1}(\cdot), \ldots, \mu_{N-1}(\cdot)\},
$
where the Borel-measurable function $\mu_{i}(\cdot):\mathcal{X} \rightarrow \mathcal{U}$, for all $i=0,\ldots,N-1$ is a general state feedback control law. 

Then the \textit{stochastic MPC} problem with initial state $x$, system parameters $\theta$ and random noise process $\boldsymbol{w} = [w_{0},w_{1},\cdots,w_{N-1}]$, is defined as 
\begin{equation}\label{eqn: value function}
\begin{split}
    \inf_{\boldsymbol{\mu}}\quad &V_{N}(x, \theta, \boldsymbol{\mu}) := \mathbb{E}_{\boldsymbol{w}}\left[\sum_{i=0}^{N-1} l(x_{i},u_{i})+V_{f}(x_{N})\right]  \\
    \text{s.t.}\quad & x_{i+1}=f(x_{i},u_{i},w_{i},\theta), \\
    & u_{i}=\mu_{i}(x_{i}), \\ 
    & x_{0} = x, \\
    & \mathrm{Pr}(x_{N} \in \mathcal{X}_{f}) \ge 1-\epsilon, \\
\end{split}
\end{equation}
where $l(x_{i},u_{i})$ is the stage-wise cost function w.r.t. the state $x_{i}$ and the control input $u_{i}$, $V_{f}(x_{N})$ is the terminal cost function, $\mathcal{X}_{f}$ is the terminal constraint, $\epsilon \in (0,1)$, and $\mathbb{E}_{\boldsymbol{w}}[\cdot]$ denotes the expectation taken w.r.t. the joint probability distribution of the random noise process $\boldsymbol{w}$.

The system parameters $\theta$, together with the random noise process $\boldsymbol{w}$, decide the joint probability distribution $p_{\theta}(\boldsymbol{x})$ of the system trajectory $\boldsymbol{x} = [x_{0},x_{1},\cdots,x_{N}]$, and thus the value function $V_{N}$.
To explore the middle ground between optimistically ignoring the distributional uncertainty of the parameters (nominal control) and pessimistically fixating on the worst-case scenario (robust control), we adopt a risk-averse perspective to quantify the impact of extreme events.
Therefore, we formulate a \textit{risk-averse MPC} problem as:
\begin{equation}\label{eqn: risk-averse MPC}
    \inf_{\boldsymbol{\mu}} \mathcal{R}[V_{N}(x, \theta, \boldsymbol{\mu})],
\end{equation}
where $\mathcal{R}$ is a risk measure taken w.r.t. the value function $V_{N}$ on the sample space $(\mathbf{x},\mathcal{F})$.

This risk-averse formulation \eqref{eqn: risk-averse MPC} is typical in many existing risk-averse autonomous systems \cite{Wang2022Risk-averseControl}, and numerous choices of the risk measure $\mathcal{R}$ can be applied, such as mean-variance, Value at Risk (VaR), and Conditional Value at Risk (CVaR) (cf. \cite{Wu2018AAsymptotics}). In particular, VaR and CVaR are two commonly used risk measures. For a random variable $Z$, $\operatorname{VaR}^{\alpha}$ is defined as the $\alpha$-quantile of $Z$, i.e., $\operatorname{VaR}^{\alpha}(Z):=\inf \{z: \mathbb{P}(Z \leq z) \geq \alpha\}$, and CVaR represents the mean deviations from quantiles, i.e., $\operatorname{CVaR}^{\alpha}(Z):=\small{\inf_{t \in \mathbb{R}}\{t+\alpha^{-1}\mathbb{E}[Z-t]_{+}\}}$. 

In this paper, we consider risk measures $\mathcal{R}$ that are coherent, and assume that the stage cost $l(x,u)$ and terminal cost $V_{f}(x)$ are finite $\forall x \in \mathcal{X}, u \in \mathcal{U}$ such that $V_{N}(x,\theta, \boldsymbol{\mu})$ is also finite for any finite prediction horizon $N$.
Then $\mathcal{R}(V_{N}(x,\theta, \boldsymbol{\mu}))$ is also finite valued (cf. Proposition 6.7 in \cite{Shapiro2009LecturesProgramming}). 
According to the duality property described in \eqref{eqn:dual risk}, \eqref{eqn: risk-averse MPC} can be represented as a \textit{distributionally robust optimization} (DRO) problem:
\begin{equation}\label{eqn: DRO original}
    \inf_{\boldsymbol{\mu}} \sup_{\mathbf{x} \in \mathcal{A}(\mathbf{x})} V_{N}(x, \theta, \boldsymbol{\mu}),
\end{equation}
where $\mathcal{A}(\mathbf{x})$ is an ambiguity set on the sample space $(\mathbf{x},\mathcal{F})$.

This denotes that for coherent risk measures, we can obtain an equivalence between optimizing a coherent risk measure and constructing the corresponding ambiguity set\cite{Ruszczynski2006OptimizationFunctions}. 
In practice, exact ambiguity set construction is usually intractable, so an approximation is often resorted to. 
A popular approach is to specify a reference probability distribution $P$, and the ambiguity set is defined as a set of probability measures in some sense close to $P$ \cite{Shapiro2021TutorialProgramming}. 
In this paper, we focus on managing the risks of parametric uncertainty and choose a set of candidate values of the parameter $\theta \in \mathcal{A}(\theta)$ and their resulting deterministic system trajectories as the reference. Thus, the probability distributions of the system trajectories polluted by noises naturally construct an ambiguity set $\mathcal{A}(\mathbf{x})$.
Then \eqref{eqn: DRO original} is equal to 
\begin{equation}\label{eqn: DRO}
    \inf_{\boldsymbol{\mu}} \sup_{\theta \in \mathcal{A}(\theta)} V_{N}(x, \theta, \boldsymbol{\mu}).
\end{equation}
For simplicity, in the rest of the paper, we denote $\mathcal{A}=\mathcal{A}(\theta)$.

In this paper, we propose an online Bayesian learning-based risk-averse MPC framework, as illustrated in Fig. \ref{fig:MPC}.
Different from the classical risk-averse MPC problem defined on a fixed ambiguity set $\mathcal{A}$, the Bayesian risk-averse MPC problem is defined on the time-varying ambiguity set $\mathcal{A}_{k}$, which is constructed based on the sequentially updated Bayesian posterior distribution of the unknown parameter $\theta$.
And since the focus of this paper is about consistency and stability, we assume that recursive feasibility is satisfied in the rest of this paper.

\begin{figure}
    \centering
    \includegraphics[width=0.48\textwidth]{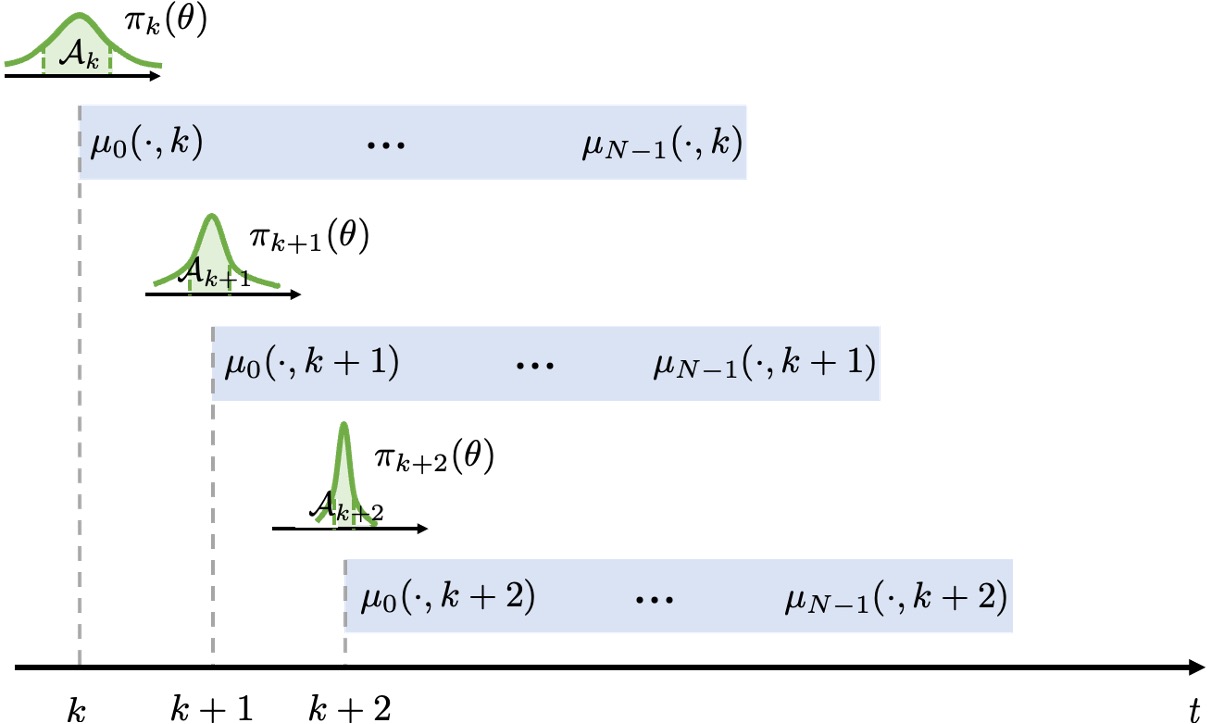}
    \caption{Bayesian risk-averse MPC framework. We use $k$ to represent the sampling time instant, and $N$ to represent the prediction horizon of MPC problems at each sampling time instant.}
    \label{fig:MPC}
\end{figure}

The formulation and properties of the proposed Bayesian risk-averse MPC framework will be derived in detail in the following three sections. 
In Section \ref{sec:Bayesian Learning}, we will first present the sequential Bayesian learning method and prove its consistency under conditionally independent measurements.
We then introduce in detail the nested formulation of risk-averse MPC in Section \ref{sec:risk-averse MPC}, and formally define and prove a general notion of \textit{risk-averse asymptotic stability}.
Finally, we close the loop between Bayesian learning and classical risk-averse MPC in Section \ref{sec:Bayesian risk-averse MPC} by developing a practical Bayesian risk-averse MPC algorithm with theoretical guarantee of stability.

\section{Bayesian Learning and Consistency}\label{sec:Bayesian Learning}

In this section, we present a Bayesian consistent online learning approach to estimate the unknown parameters in the system dynamics.  
We first derive a sequential Bayesian learning update rule in Subsection \ref{subsec:Bayesian Learning}, and then prove its consistency under conditionally independent measurements in Subsection \ref{subsec:Bayesian Consistency}.

\subsection{Sequential Bayesian Learning}\label{subsec:Bayesian Learning}
Let $x_{0:k}$ be the states from time $0$ to $k$, and $u_{0:k}$ be the corresponding control inputs from time $0$ to $k$. Since the system is Markovian, 
\begin{equation*}
    \mathrm{Pr}(x_{k}|\theta,x_{0:k-1},u_{0:k-1})=\mathrm{Pr}(x_{k}|\theta,x_{k-1},u_{k-1}).
\end{equation*}

Define the Bayesian prior $\pi_0$ on $(\Theta, \mathcal{B}_{\Theta})$, where $\mathcal{B}_{\Theta}$ is the Borel $\sigma$-algebra on $\Theta$. Define $\pi_{k}(\theta)=\mathrm{Pr}(\theta|x_{0:k},u_{0:k})$ as the posterior distribution of $\theta$ at time $k$. Therefore,
\begin{equation*}
    \pi_{k}(\theta)=\mathrm{Pr}(\theta|x_{0:k},u_{0:k})= \mathrm{Pr}(\theta|x_{0:k},u_{0:k-1}),
\end{equation*}
followed from the fact that the control input $u_{k}$ does not affect the information on $\theta$ until an observation of the new state $x_{k+1}$ is taken.

According to the Bayes' rule,
\begin{equation*}
\begin{split}
&\mathrm{Pr}(\theta|x_{0:k},u_{0:k}) = \mathrm{Pr}(\theta|x_{0:k},u_{0:k-1})\\
=&\frac{\mathrm{Pr}(\theta,x_{k}|x_{0:k-1},u_{0:k-1})}{\mathrm{Pr}(x_{k}|x_{0:k-1},u_{0:k-1})} \\
=&\frac{\mathrm{Pr}(x_{k}|\theta,x_{0:k-1},u_{0:k-1})\mathrm{Pr}(\theta|x_{0:k-1},u_{0:k-1})}{\mathrm{Pr}(x_{k}|x_{0:k-1},u_{0:k-1})} \\
=&\frac{\mathrm{Pr}(x_{k}|\theta,x_{k-1},u_{k-1})\mathrm{Pr}(\theta|x_{0:k-1},u_{0:k-1})}{\int\mathrm{Pr}(x_{k}|\theta,x_{k-1},u_{k-1}) \mathrm{Pr}(\theta | x_{0:k-1}, u_{0:k-1})d\theta}.
\end{split}
\end{equation*}

Define a \textit{state transition stochastic kernel} $q(\cdot)$ that satisfies $\int_B q(x_k;\theta, x_{k-1}, u_{k-1})dx_{k} = \mathrm{Pr}(B|\theta,x_{k-1},u_{k-1})$, where $B$ is an arbitrary Borel-measurable set (cf. Section 8.1 in \cite{Bertsekas1996StochasticCase}). 
Since the disturbance $w_{k-1}$ admits the form of a probability density function, the transition kernel $q(x_k;\theta, x_{k-1}, u_{k-1})=\mathrm{Pr}(f(x_{k-1},u_{k-1},w_{k-1},\theta)|\theta,x_{k-1},u_{k-1})$ is also a probability density function, which is determined by the Borel-measurable transition function $f(\cdot)$ and the probability density function of disturbance $w_{k-1}$. 
Therefore, the posterior distribution of $\theta$ is updated as 
\begin{equation}\label{eqn:posterior distribution}
    \pi_{k}(\theta)=\frac{q(x_{k};\theta,x_{k-1},u_{k-1})}{\int q(x_{k};\theta,x_{k-1},u_{k-1}) \pi_{k-1}(\theta)d\theta }\pi_{k-1}(\theta).
\end{equation}

For the majority of the systems in the field of control, the noise is assumed to be Gaussian white noise, which naturally satisfies the following assumption.

\begin{assumption}
The transition kernel $q(x';\theta,x,u)$ is continuously differentiable, strictly positive, and has bounded first-order derivative in $x'$.
\end{assumption}

With the above assumption, we can easily show that the transition kernel $q(x';\theta,x,u)$ is bounded, then the integration in \eqref{eqn:posterior distribution} is non-zero and finite since $\int \pi(\theta) d\theta = 1$. Thus the posterior distribution is well-defined. 


\subsection{Consistency of Sequential Bayesian Learning}\label{subsec:Bayesian Consistency}

\begin{definition}[Bayesian Consistency \cite{Vaart1998AsymptoticStatistics}]
The Bayesian estimator is \textbf{(strongly) consistent} if $\pi_k$, the posterior distribution of $\theta$, converges to the degenerated distribution $\delta_{\theta^{\ast}}$ that concentrates on the true parameter value $\theta^{*}$, with probability 1 (w.p.1.). 
\end{definition}

\noindent \textbf{Remark.} Due to the technical challenges of analyzing the continuous parameter space, we analyze the consistency of the Bayesian estimator by assuming that the parameter space $\Theta$ is discrete but consists of an infinite number of candidates, which can approximate the continuous parameter space with arbitrary precision.

To derive the conditions that ensure Bayesian consistency under conditionally independent measurements, we first introduce the concept of observational distinguishability (or identifiability, cf. Definition 5.2 in \cite{Lehmann2006TheoryEstimation} and Assumption 3.14 in \cite{Liu2024BayesianData}).
Let $\mathcal{P}(\eta)=\{\mathrm{Pr}(\cdot|\theta, \eta),\theta \in \Theta\}$ be a statistical model with parameter space $\Theta$ for a specific context $\eta$. 

\begin{definition}[Observational Distinguishability]
The elements within $\mathcal{D} \subseteq \Theta$ are \textbf{observationally non-distinguishable} from $\{\mathcal{P}(\eta_{m}), m \in \mathcal{M} \subseteq \mathbb{N}\}$ where $\mathbb{N}$ is the set of natural numbers, if there exist non-zero scalars $c_{1}, c_{2}, \cdots$ such that $\sum_{\theta_{i}\in \mathcal{D}} c_{i}\mathrm{Pr}(\cdot|\theta_{i}, \eta_{m}) = 0, \forall m \in \mathcal{M}$.
Otherwise, they are (at least partially) \textbf{observationally distinguishable}.
Furthermore, they are \textbf{strongly observationally distinguishable} if each element is observationally distinguishable to others, i.e., if $\sum_{\theta_{i}\in \mathcal{D}} c_{i}\mathrm{Pr}(\cdot|\theta_{i}, \eta_{m}) = 0, \forall m \in \mathcal{M}$, then $c_{1}=c_{2}=\cdots=0$.
\end{definition}

From the definitions above, the following properties are straightforward:
\begin{proposition}[Observational Distinguishability with Extended Contexts]\label{prop:extened contexts}
Let $\mathcal{M}_{1} \subseteq \mathcal{M}_{2}$. If the elements within $\mathcal{D} \subseteq \Theta$ are observationally non-distinguishable from $\{\mathcal{P}(\eta_{m}), m \in \mathcal{M}_{2}\}$, then they are also observationally non-distinguishable from $\{\mathcal{P}(\eta_{m}), m \in \mathcal{M}_{1}\}$. 
On the contrary, if the elements within $\mathcal{D} \subseteq \Theta$ are (strongly) observationally distinguishable from $\{\mathcal{P}(\eta_{m}), m \in \mathcal{M}_{1}\}$, then they are also (strongly) observationally distinguishable from $\{\mathcal{P}(\eta_{m}), m \in \mathcal{M}_{2}\}$. 
\end{proposition}

Then we formally define the concepts of blind zone and blind region within the parameter space according to their properties of observational distinguishability.

\begin{definition}[Blind Zone]\label{def:BZ}
A set $\mathcal{BZ}(\eta_{m}, m \in \mathcal{M}) \subseteq \Theta$ is a \textbf{blind zone} of $\{\mathcal{P}(\eta_{m}), m \in \mathcal{M}\}$ if it contains at least two elements, and 
\begin{itemize}
    \item[(i)] The elements within $\mathcal{BZ}(\eta_{m}, m \in \mathcal{M})$ are observationally non-distinguishable from $\{\mathcal{P}(\eta_{m}), m \in \mathcal{M}\}$;
    \item[(ii)] Any element outside $\mathcal{BZ}(\eta_{m}, m \in \mathcal{M})$ is observationally distinguishable to the elements within $\mathcal{BZ}(\eta_{m}, m \in \mathcal{M})$ from $\{\mathcal{P}(\eta_{m}), m \in \mathcal{M}\}$.
\end{itemize}
\end{definition}

\begin{proposition}[Combination of Blind Zones]\label{prop:comb BZ}
Let $\mathcal{BZ}(\eta_{1})$ and $\mathcal{BZ}(\eta_{2})$ be blind zones of $\{\mathcal{P}(\eta_{1})\}$ and $\{\mathcal{P}(\eta_{2})\}$ respectively. The blind zone of $\{\mathcal{P}(\eta_{1}),\mathcal{P}(\eta_{2})\}$, $\mathcal{BZ}(\eta_{1}) \otimes \mathcal{BZ}(\eta_{2})$, is $\mathcal{BZ}(\eta_{1}) \cap \mathcal{BZ}(\eta_{2})$ if the elements within $\mathcal{BZ}(\eta_{1}) \cap \mathcal{BZ}(\eta_{2})$ are observationally non-distinguishable from $\{\mathcal{P}(\eta_{1}),\mathcal{P}(\eta_{2})\}$; otherwise, it is $\emptyset$. 
\end{proposition}

\begin{proof}
According to Definition \ref{def:BZ}, any element within $\Theta\setminus \mathcal{BZ}(\eta_{1})$ is observationally distinguishable to the elements within $\mathcal{BZ}(\eta_{1})$ from $\{\mathcal{P}(\eta_{1})\}$. 
Then from Proposition \ref{prop:extened contexts}, it is also observationally distinguishable to the elements within $\mathcal{BZ}(\eta_{1})$ from $\{\mathcal{P}(\eta_{1}), \mathcal{P}(\eta_{2})\}$.
Similarly, any element within $\Theta\setminus \mathcal{BZ}(\eta_{2})$ is observationally distinguishable to the elements within $\mathcal{BZ}(\eta_{2})$ from $\{\mathcal{P}(\eta_{1}), \mathcal{P}(\eta_{2})\}$.
Therefore, any element within $\Theta\setminus (\mathcal{BZ}(\eta_{1})\cap \mathcal{BZ}(\eta_{2}))$ is observationally distinguishable to the elements within $\mathcal{BZ}(\eta_{1})\cap \mathcal{BZ}(\eta_{2})$ from $\{\mathcal{P}(\eta_{1}), \mathcal{P}(\eta_{2})\}$.
\end{proof}

\begin{definition}[Blind Region]\label{def:BR}
The \textbf{blind region} $\mathcal{BR}(\eta_{m}, m \in \mathcal{M})$ of $\{\mathcal{P}(\eta_{m}), m \in \mathcal{M}\}$ is defined as the union of all possible blind zones,  
$
    \mathcal{BR}(\eta_{m}, m \in \mathcal{M})= \bigcup_{i} \mathcal{BZ}_{i}(\eta_{m}, m \in \mathcal{M}),
$
where $\mathcal{BZ}_{i}(\eta_{m}, m \in \mathcal{M}) \subseteq \Theta$ is a sequence of disjoint blind zone of $\{\mathcal{P}(\eta_{m}), m \in \mathcal{M}\}$.
\end{definition}

\begin{proposition}[Combination of Blind Regions]\label{prop:comb BR}
Let the blind regions of $\{\mathcal{P}(\eta_{1})\}$ and $\{\mathcal{P}(\eta_{2})\}$ be $\mathcal{BR}(\eta_{1})=\bigcup_{i}\mathcal{BZ}_{i}(\eta_{1})$ and $\mathcal{BR}(\eta_{2})=\bigcup_{j}\mathcal{BZ}_{j}(\eta_{2})$ respectively.
The blind region of $\{\mathcal{P}(\eta_{1}),\mathcal{P}(\eta_{2})\}$,
$\mathcal{BR}(\eta_{1}) \oplus \mathcal{BR}(\eta_{2})$, is $\bigcup_{i,j} \mathcal{BZ}_{i}(\eta_{1}) \otimes \mathcal{BZ}_{j}(\eta_{2})$. 
\end{proposition}

\begin{proof}
According to Definition \ref{def:BR}, the elements within $\Theta\setminus \mathcal{BR}(\eta_{1})$ are strongly observationally distinguishable from $\{\mathcal{P}(\eta_{1})\}$. Then from Proposition \ref{prop:extened contexts}, they are also strongly observationally distinguishable from $\{\mathcal{P}(\eta_{1}),\mathcal{P}(\eta_{2})\}$.
Similarly, the elements within $\Theta\setminus \mathcal{BR}(\eta_{2})$ are strongly observationally distinguishable from $\{\mathcal{P}(\eta_{1}),\mathcal{P}(\eta_{2})\}$.
Therefore, $\Theta \setminus (\mathcal{BR}(\eta_{1}) \cap \mathcal{BR}(\eta_{2}))$ cannot be a blind zone of $\{\mathcal{P}(\eta_{1}),\mathcal{P}(\eta_{2})\}$.

For the elements within $\mathcal{BR}(\eta_{1}) \cap \mathcal{BR}(\eta_{2})$, we take all possible combinations between the blind zones $\mathcal{BZ}_{i}(\eta_{1})$ of $\{\mathcal{P}(\eta_{1})\}$ and the blind zones $\mathcal{BZ}_{j}(\eta_{2})$ of $\{\mathcal{P}(\eta_{2})\}$ to obtain the blind zones of $\{\mathcal{P}(\eta_{1}),\mathcal{P}(\eta_{2})\}$ according to Proposition \ref{prop:comb BZ}. Therefore, $\mathcal{BR}(\eta_{1}) \oplus \mathcal{BR}(\eta_{2})$ is the union of all possible blind zones of $\{\mathcal{P}(\eta_{1}),\mathcal{P}(\eta_{2})\}$, and according to Definition \ref{def:BR}, it is the blind region of $\{\mathcal{P}(\eta_{1}),\mathcal{P}(\eta_{2})\}$.
\end{proof}

Now the Bayesian consistency conditions can be derived based on the blind region of the observed dataset during the system trajectory:

\begin{theorem}[Bayesian Consistency Conditions]\label{thm: consistency}
Suppose 
\begin{itemize}
    \item[(i)] The prior distribution $\pi_{0}$ has non-zero probability at $\theta^{\ast}$,
    \item[(ii)] (Combinationally Identifiable) There exist a collection of convergent subsequences $\{\{x^{m}_{k},u^{m}_{k}\}_{k=1:\infty}, m \in \mathcal{M} \subseteq \mathbb{N}\}$ during the system trajectory, such that the combination of blind regions of $\mathcal{P}(\eta_{m})$ in the converged context $\eta_{m}=[x^{m}_{\infty},u^{m}_{\infty}]$ is empty, i.e., 
    $$\mathcal{BR}(\eta_{1}) \oplus \mathcal{BR}(\eta_{2}) \oplus \cdots \oplus \mathcal{BR}(\eta_{m}) \oplus \cdots = \emptyset,$$
    where $\mathcal{BR}(\eta_{m})$ is the blind region of $\mathcal{P}(\eta_{m})$,
\end{itemize}
then the posterior distribution of $\theta$ converges to the degenerated distribution $\delta_{\theta^{\ast}}$, i.e. $\lim_{k\rightarrow \infty} \pi_{k}=\delta_{\theta^{\ast}}$, w.p.1. 
\end{theorem}

\begin{proof}
The proof can be found in Appendix \ref{appendix:consistancy}.
\end{proof}

\noindent \textbf{Remark 1.} Theorem \ref{thm: consistency} provides explicit conditions that guarantee parameter convergence in a stochastic setting, which can be considered as an analogy to the \textit{persistency of excitation}\cite{Willems2005AExcitation} in a deterministic setting.
It can function as a general tool to verify the consistency guarantees of general online learning-based control algorithms on any specific system.

\noindent \textbf{Remark 2.} The introduction of the concepts of blind zone and region, as well as the operations to combine them, provides an intuitive way to interpret the consistency conditions. It further offers insights to design control policies that explicitly balance exploration and exploitation, for instance, the combined blind regions monitored online can serve as the interested states to explore in active inference-based methods \cite{Smith2022AData}.

\section{Risk-averse MPC and Stability}\label{sec:risk-averse MPC}

In this section, we demonstrate a risk-averse MPC framework to account for the uncertainty encoded by probability distributions. The nested formulation of risk-averse MPC is first introduced in detail in Subsection \ref{subsec:DP}. Then in Subsection \ref{subsec:risk-averse stability}, we introduce a general notion of \textit{risk-averse asymptotic stability}, and rigorously prove the risk-averse Lyapunov stability theorem and risk-averse MPC stability conditions.

\subsection{Dynamic Programming of Risk-averse MPC}\label{subsec:DP}

The entanglement of parametric uncertainty and disturbances renders the original formulation (\ref{eqn: DRO}) intractable. We address this via a more cautious, \textit{nested} formulation (\ref{eqn: nested DRO}) that admits efficient solution via DP.

Consider a closed-loop dynamic system controlled by a certain control policy $\boldsymbol{\mu}$:
\begin{equation}\label{eqn: closed-loop system}
    x_{k+1}=f(x_{k},\mu_{0}(x_{k}),w_{k},\theta).
\end{equation}
Let $\phi(i;x,\theta,\boldsymbol{\mu},\boldsymbol{w})$ denote the solution $x_{i}$ of (\ref{eqn: closed-loop system}) at time $i$ if the initial state at time 0 is $x$, the control at $(x,i)$ is $\mu_{0}(x)$, and the disturbance sequence is $\boldsymbol{w}$. 

Since the process noises are assumed to be independent, the value function can be written as a nested formulation:
\begin{equation*}
    V_{N} (x, \theta, \boldsymbol{\mu}) = \mathbb{E}_{w_{0}}\left[  
    \mathbb{E}_{w_{1}}\left[  
    \cdots \mathbb{E}_{w_{N-1}}\left[ \cdot  \right]
    \right]
    \right].
\end{equation*}

It follows that 
\begin{equation}\label{eqn: nested inequality}
    \sup_{\theta \in \mathcal{A}} V_{N} (x, \theta, \boldsymbol{\mu}) \le \sup_{\theta \in \mathcal{A}} \mathbb{E}_{w_{0}}\left[  
    \sup_{\theta \in \mathcal{A}} \mathbb{E}_{w_{1}}\left[  
    \cdots \sup_{\theta \in \mathcal{A}} \mathbb{E}_{w_{N-1}}\left[ \cdot  \right]
    \right]
    \right].
\end{equation}
Strict inequality might hold due to that, on the right-hand side of (\ref{eqn: nested inequality}), the maxima w.r.t. $\theta \in \mathcal{A}$ can depend on the realization of the noise process \cite{Shapiro2009LecturesProgramming}.

For any scalar function $g:\phi(0:i;x,\theta,\boldsymbol{\mu},\boldsymbol{w})\rightarrow \mathbb{R}$, define the operators $\gamma_i[\cdot]$ and $\bar{\gamma_i}[\cdot]$ as
$$\gamma_{i}[g]:=\sup_{\theta \in \mathcal{A}} \mathbb{E}_{w_{i}}[g(\phi(0:i;x,\theta,\boldsymbol{\mu},\boldsymbol{w}))],$$ 
$$\bar{\gamma}_{i}[g]:=\gamma_{0}[\gamma_{1}[\cdots \gamma_{i}[g(\phi(0:i;x,\theta,\boldsymbol{\mu},\boldsymbol{w}))]]].$$
Then the right-hand side of \eqref{eqn: nested inequality} leads to the following \textit{nested} formulation of the risk-averse MPC problem \cite{Rawlings2017ModelDesign}:
\begin{equation}\label{eqn: nested DRO}
    \begin{split}
        \inf_{u_{0}} & l(x_{0},u_{0}) + 
        \gamma_{0}  \left[ \inf_{u_{1}} l(x_{1},u_{1}) 
        + \gamma_{1} \Biggl[ \dots \right.  \\
        &\left. \left. + \inf_{u_{N-1}} l(x_{N-1},u_{N-1}) + \gamma_{N-1} \left[ V_{f}(x_{N})\right] \right] \right].
    \end{split}
\end{equation}

If the process noises $\boldsymbol{w}$ are assumed to be i.i.d, then we can omit the subscript $i$ in $\gamma_{i}[\cdot]$.
For simplicity, in the remainder of this paper, we use $\gamma[\cdot]$ as an abbreviation of $\gamma_{i}[\cdot]$.

The nested formulation leads to the DP equations. 
The DP recursions are given as
\begin{equation}\label{eqn:DP value}
    \begin{split}
        V_{i}^{\ast}(x)= \inf_{u} l(x,u) + \gamma [V_{i-1}^{\ast}(f(x,u,w,\theta))],
    \end{split}
\end{equation}
\begin{equation}\label{eqn:DP policy}
    \begin{split}
        \kappa_{i}(x)= \arg \inf_{u} l(x,u) + \gamma [V_{i-1}^{\ast}(f(x,u,w,\theta))],
    \end{split}
\end{equation}
with boundary conditions
\begin{equation}
    V_{0}^{\ast}(x)=V_{f}(x),
\end{equation}
where $i$ denotes time to go so that $\kappa_{i}(\cdot):=\mu_{N-i}(\cdot)$.

\subsection{Stability of Risk-averse MPC}\label{subsec:risk-averse stability}

In order to define risk-averse notions of stability, we first introduce an appropriate notion of invariance set from a risk-averse perspective.

\begin{definition}[Risk-averse Positive Invariant]
A set $S$ is Risk-averse Positive Invariant (RAPI) for (\ref{eqn: closed-loop system}) if $x \in S$ implies that $\mathbb{E}_{w}[f(x,\mu_{0}(x),w,\theta)] \in S$ for any $\theta \in \mathcal{A}$.
\end{definition}

To establish stability we make use of Lyapunov theorems that are defined in terms of the function classes $\mathcal{K}$, $\mathcal{K}_{\infty}$, $\mathcal{L}$ and $\mathcal{KL}$:
\begin{itemize}
    \item A function $\alpha : \mathbb{R}_{\ge 0} \rightarrow \mathbb{R}_{\ge 0}$ is said to be of class-$\mathcal{K}$ ($\alpha \in \mathcal{K}$) if it is continuous, zero at zero, and strictly increasing. 
    \item A function $\alpha : \mathbb{R}_{\ge 0} \rightarrow \mathbb{R}_{\ge 0}$ is said to be of class-$\mathcal{K}_{\infty}$ ($\alpha \in \mathcal{K}_{\infty}$) if $\alpha \in \mathcal{K}$ and, in addition, $\lim_{s \rightarrow \infty} \alpha(s) = \infty$.
    \item A function $\sigma : \mathbb{R}_{\ge 0} \rightarrow \mathbb{R}_{\ge 0}$ is said to be of class-$\mathcal{L}$ ($\sigma \in \mathcal{L}$) if it is continuous, strictly decreasing, and $\lim_{s \rightarrow \infty} \sigma(s) = 0$.
    \item A function $\beta : \mathbb{R}_{\ge 0} \times \mathbb{I}_{\ge 0} \rightarrow \mathbb{R}_{\ge 0}$ is said to be of class-$\mathcal{KL}$ ($\beta \in \mathcal{KL}$) if it is class-$\mathcal{K}$ in its first argument and class-$\mathcal{L}$ in its second argument. 
\end{itemize}

Now we can define the risk-averse stability notion as follows.

\begin{definition}[Risk-averse Asymptotically Stable]
Suppose $\mathcal{X}$ is RAPI for (\ref{eqn: closed-loop system}). The origin is Risk-averse Asymptotically Stable (RAAS) for (\ref{eqn: closed-loop system}) in $\mathcal{X}$ if there exists a $\mathcal{KL}$ function $\beta(\cdot)$ such that, for each $x \in \mathcal{X}$ 
\begin{equation}
    \bar{\gamma}_{i-1}[|\phi(i;x,\theta,\boldsymbol{\mu},\boldsymbol{w})|] \le \beta (|x|,i), \quad \forall i \in \mathbb{I}_{\ge 0}.
\end{equation}
\end{definition}

RAAS entails that the origin is stable not only for the system with nominal parameters, but also for those systems with parameters in the ambiguity set of the risk measure. Since all coherent risk measures are lower bounded by the expectation, RAAS is a stronger notion of stability compared to the stability of classical stochastic control. 

Similar to the classical Lyapunov stability theorem, we can derive a risk-averse Lyapunov function which leads to RAAS.

\begin{theorem}[Risk-averse Lyapunov Stability Theorem]\label{thm:RA Lyapunov stability}
Suppose $\mathcal{X}$ is RAPI for (\ref{eqn: closed-loop system}). If there exists a Lyapunov function $V: \mathbb{R}^{n_{x}} \rightarrow \mathbb{R}_{\ge 0}$ in $\mathcal{X}$ for (\ref{eqn: closed-loop system}) such that for any $x \in \mathcal{X}$:
\begin{itemize}
    \item[(i)] $\alpha_{1}(|x|) \le V(x) \le \alpha_{2}(|x|)$, where $\alpha_{1},\alpha_{2} \in \mathcal{K}_{\infty}$,
    \item[(ii)] $\gamma[V(f(x,\mu_{0}(x),w,\theta))] - V(x) \le -\rho(x)$, where $\rho$ is a continuous positive definite function,
\end{itemize}
then the origin is RAAS in $\mathcal{X}$ for (\ref{eqn: closed-loop system}).
\end{theorem}

\begin{proof}
The proof can be found in Appendix \ref{appendix:RA Lyapunov stability}.
\end{proof}

We may now state conditions on the stage cost $l$ and the terminal cost $V_{f}$ that ensure RAAS for the risk-averse MPC-controlled system, where the control policy $\mu_{0}(\cdot)$ is chosen as the optimal MPC policy $\kappa_{N}(\cdot)$ obtained by \eqref{eqn:DP policy}.

\begin{theorem}[Risk-averse MPC Stability Conditions]\label{thm:RA MPC stability}
Suppose that,
\begin{itemize}
    \item[(i)] There exists a $\mathcal{K}_{\infty}$ function $\alpha_{1}(\cdot)$ such that $l(x,u) \ge \alpha_{1}(|x|)$, $\forall x \in \mathcal{X}$,
    \item[(ii)] There exists a $\mathcal{K}_{\infty}$ function $\alpha_{2}(\cdot)$ such that $V_{N}^{\ast}(x) \le \alpha_{2}(|x|)$, $\forall x \in \mathcal{X}$,
    \item[(iii)] For all $x \in \mathcal{X}_{f}$, there exists a $u$ such that $\gamma[f(x,u,w,\theta)] \in \mathcal{X}_{f}$, and $\gamma[V_{f}(f(x,u,w,\theta))]-V_{f}(x) \le -l(x,u)$,
\end{itemize}
then the origin is RAAS in $\mathcal{X}$ for the risk-averse MPC-controlled system $x_{k+1}=f(x_{k},\kappa_{N}(x_{k}),w_{k},\theta)$.
\end{theorem}

\begin{proof}
The proof can be found in Appendix \ref{appendix:RA MPC stability}.
\end{proof}

\noindent \textbf{Remark.} The stabilizing conditions in Theorem \ref{thm:RA MPC stability} can be implemented in a variety of ways. For example, in the \textit{Linear Quadratic Regulator} (LQR), the stage cost and terminal cost are often chosen as $l(x,u)=\frac{1}{2}(x^{\mathsf{T}}Qx+u^{\mathsf{T}}Ru)$ and $V_{f}(x)=\frac{1}{2}x^{\mathsf{T}}\Pi x$, where $Q,R$ are positive definite matrices, $\Pi$ is the solution to the steady-state Riccati equation (cf. \cite{Rawlings2017ModelDesign}).

\section{Bayesian Risk-averse Model Predictive Control}\label{sec:Bayesian risk-averse MPC}

In this section, we develop a practical Bayesian risk-averse MPC algorithm to capacitate computationally tractable implementation with theoretical guarantees of consistency and stability. 
The major challenges for practical implementation come from two aspects: one is the calculation of the Bayesian posterior update, and the other is the optimization of the multi-stage value function. 
We address the first challenge by utilizing sampling-based approaches to approximate the sequential Bayesian update in Subsection \ref{subsec:particle filter}. 
A time-varying ambiguity set is then constructed based on the estimated Bayesian posterior distribution, which will be described in detail in Subsection \ref{subsec:ambiguity set}.
To deal with the challenge of multi-stage stochastic optimization, we first parameterize candidate control policies in Subsection \ref{subsec:optimal policy}, and then derive an alternative sub-optimal control policy in Subsection \ref{subsec:sub-optimal policy}, which is computationally efficient but still enables the system to achieve asymptotic stability.

\subsection{Computational Bayesian Estimator}\label{subsec:particle filter}

When the parameter space $\Theta$ is continuous, or infinitely countable, it is often difficult to obtain an analytical form of the posterior distribution of $\theta$. Since the analytic solution is intractable, computational algorithms are necessary. In practice, we consider utilizing the well-known sequential sampling-based approaches, \textit{Particle Filtering} (PF)\cite{Doucet2010ALater}, to  approximate the sequential Bayesian update in \eqref{eqn:posterior distribution}.

\noindent \textbf{Remark.} In practice, the hyperparameters of particle filter-based approaches should be tailored to the specific stochastic models of interest. In this paper, we assume that an accurate approximation of the Bayesian update can be achieved through appropriate choices of hyperparameters. However, the sample complexity, specifically the number of particles required, is beyond the scope of this paper.

\subsection{Ambiguity Set}\label{subsec:ambiguity set}

In this paper, we consider utilizing \textit{credible interval} of the parameter distribution $\pi$ as the ambiguity set $\mathcal{A}$. 
In Bayesian statistics, a \textit{credible interval} is an interval within which an unobserved parameter value falls with a particular probability. It is an interval in the domain of a posterior probability distribution or a predictive distribution. The generalization to multivariate problems is the \textit{credible region}. Specifically, the \textit{Credible Interval} (CI) with a credible level $l\in [0,1]$, $\mathcal{C}^{l}$, is defined as a continuous subset such that the probability that the value of the random variable falls within that subset is $l$, i.e., $\mathrm{Pr}(\theta \in \mathcal{C}^{l}) = l.$

For a given posterior distribution $\pi$ and credible level $l$, $\mathcal{C}^{l}$ is not unique. There are typically two types of Bayesian CIs: (1) equal tail interval (ETI); and (2) highest posterior density interval (HPDI). The ETI chooses the interval where the probability of being below the interval is as likely as being above it: $\mathcal{C}^{l}=[q_{l},q_{u}]$ such that $\mathrm{Pr}(\theta < q_{l})=\mathrm{Pr}(\theta > q_{u})=(1-l)/2$. 
The HPDI chooses the narrowest set such that the posterior density for every point in this set is higher than the posterior density for any point outside of this set: $\mathrm{Pr}(\theta)\ge \mathrm{Pr}(\theta')$, for all $\theta \in \mathcal{C}^{l}, \theta' \notin \mathcal{C}^{l}$. For example, for an unimodal distribution, HPDI choose the values of highest probability density including the mode (the maximum a posteriori).

In this paper, since the posterior distribution of $\theta$ is updated sequentially, we also sequentially construct the time-varying ambiguity set $\mathcal{A}_{k}$ based on the estimated parameter distribution $\pi_{k}$. 
Let $\mathcal{C}^{l}_{k}$ be the CI for $\pi_{k}$ with a predefined credible level $l$.
If the Bayesian consistency conditions are satisfied, according to Theorem \ref{thm: consistency}, the posterior distribution of $\theta$ obtained by the Bayesian estimator converges to the degenerated distribution $\delta_{\theta^{\ast}}$ that concentrates on the true parameter value $\theta^{\ast}$, w.p.1. As a result, it is obvious to see that the ambiguity set $\mathcal{C}_{k}^{l}$ will also converge to a singleton that only contains $\theta^{\ast}$, i.e. $\mathcal{C}_{\infty}^{l}=\{\theta^{\ast}\}$, w.p.1. Then it is possible to find a subsequence of ambiguity sets $\mathcal{A}_{0:k}$ such that 
\begin{equation*}
    \begin{split}
        \mathcal{A}_{0} \supseteq \mathcal{A}_{1} \supseteq \cdots \supseteq \mathcal{A}_{k},\quad \text{and} \quad
        \lim_{k\rightarrow \infty}\mathcal{A}_{k} = \{\theta^{\ast}\}.
    \end{split}
\end{equation*}

Therefore, we assume that the credible intervals $\mathcal{C}^{l}_{k}$ contain $\theta^{\ast}$ within their interiors, which is often satisfiable in practice.
To facilitate the stability of the time-varying risk-averse MPC problem, we consider constructing a time-varying ambiguity set $\mathcal{A}_{k}$ which is forward shrinking. More specifically, let 
\begin{equation}\label{eqn:ambiguity set}
    \mathcal{A}_{k}=\left\{
    \begin{array}{ll}
    \mathcal{C}^{l}_{k} & \text{if}\quad \mathcal{C}^{l}_{k} \subseteq \mathcal{A}_{k-1} \quad \text{or}\quad k=0\\
    \mathcal{A}_{k-1} & \text{if}\quad \mathcal{C}^{l}_{k} \not\subseteq \mathcal{A}_{k-1}
    \end{array}
    \right. .
\end{equation}

Define the radius of the ambiguity set $\mathcal{A}_{k}$ as
\begin{equation*}
    \varepsilon_{k} = \frac{1}{2}\sup||\theta_{1}-\theta_{2}||_2,\quad \forall \theta_{1},\theta_{2} \in \mathcal{A}_{k}.
\end{equation*}
Then it is apparent that $\{\varepsilon_{k}\}$ is non-increasing and $\lim_{k \rightarrow \infty}\varepsilon_{k}=0$, w.p.1.

\subsection{Parameterized Optimal Control Policy}\label{subsec:optimal policy}

Different from the classical risk-averse MPC problem that is defined on fixed ambiguity set $\mathcal{A}$, the Bayesian risk-averse MPC problem is defined on the time-varying ambiguity set $\mathcal{A}_{k}$.
Therefore, the value function $V_{N}(\cdot)$ also becomes time-varying, and we actually need to solve a sequence of time-varying risk-averse MPC problems:
\begin{equation}
    V_{N}^{\ast}(x,k):=\inf_{\boldsymbol{\mu}} \sup_{\theta \in \mathcal{A}_{k}} V_{N}(x, \theta, \boldsymbol{\mu}).  
\end{equation}

The decision variable in DP, is a sequence $\boldsymbol{\mu}:=\{ \mu_{0}(\cdot), \mu_{1}(\cdot), \ldots, \mu_{N-1}(\cdot)\}$ of control laws, each of which is an arbitrary function of the state $x$, and thus is too complex for online optimization. Therefore, we consider a parameterized control policy $\boldsymbol{\mu}(\boldsymbol{v}):=\{ \mu(\cdot,v_{0}), \mu(\cdot,v_{1}), \ldots, \mu(\cdot,v_{N-1})\}$ which is parameterized by a sequence of parameters $\boldsymbol{v}=\{v_{0},v_{1},\ldots,v_{N-1}\}$.

With this parameterization, the time-varying risk-averse MPC problem becomes
\begin{equation}\label{eqn: parameterized formulation}
    V_{N}^{\ast}(x,k):=\inf_{\boldsymbol{v}} \sup_{\theta \in \mathcal{A}_{k}} V_{N}(x, \theta, \boldsymbol{\mu}(\boldsymbol{v})).   
\end{equation}

By solving the DP recursions with the parameterized control policy and the ambiguity set $\mathcal{A}_{k}$ at the current time $k$, 
\begin{equation}\label{eqn:k DP value}
    \begin{split}
        V_{i}^{\ast}(x,k)= &\inf_{v} l(x,\mu(x,v)) \\
        &+ \sup_{\theta \in \mathcal{A}_{k}} \mathbb{E}_{w} [V_{i-1}^{\ast}(f(x,\mu(x,v),w,\theta),k)],
    \end{split}
\end{equation}
\begin{equation}\label{eqn:k DP policy}
    \begin{split}
        \kappa_{i}(x,k)= &\mu (x, \arg \inf_{v} l(x,\mu(x,v)) \\
        &+ \sup_{\theta \in \mathcal{A}_{k}} \mathbb{E}_{w} [V_{i-1}^{\ast}(f(x,\mu(x,v),w,\theta),k)]),
    \end{split}
\end{equation}
we can recursively find the optimal value of the decision variable $\boldsymbol{v}$, denoted as $\boldsymbol{v}^{\ast}(k)$. Then the corresponding optimal control policy $\boldsymbol{\mu}^{\ast}(x,k):=\boldsymbol{\mu}(x, \boldsymbol{v}^{\ast}(k))$ and value function $V_{N}^{\ast}(x,k)$ can be obtained in the sequel.

Since MPC takes a receding horizon approach where only the first step of the control policy is employed, the implicit MPC law is given as $\mu^{\ast}_{0}(x,k)=\mu(x,v^{\ast}_{0}(k))$, with $v^{\ast}_{0}(k)$ the first element of the optimal decision variable sequence $\boldsymbol{v}^{\ast}(k)$.

Next, we state conditions on the stage cost $l$ and the terminal cost $V_{f}$ that ensure asymptotically stable for the Bayesian risk-averse MPC-controlled system, where the time-varying control policy $\kappa_{N}(\cdot,k)$ is obtained by \eqref{eqn:k DP policy}.

\begin{theorem}[Bayesian Risk-averse MPC Stability Conditions]\label{thm:Bayesian RA MPC stability}
Suppose that the Bayesian consistency conditions are satisfied, and
\begin{itemize}
    \item[(i)] There exists a $\mathcal{K}_{\infty}$ functions $\alpha_{1}(\cdot)$ such that $l(x,u) \ge \alpha_{1}(|x|)$, $\forall x \in \mathcal{X}$,
    \item[(ii)] There exists a $\mathcal{K}_{\infty}$ function $\alpha_{2}(\cdot)$ such that $V_{N}^{\ast}(x,k) \le \alpha_{2}(|x|),\forall k \in \mathbb{I}_{\ge 0}$, $\forall x \in \mathcal{X}$,
    \item[(iii)] For all $x \in \mathcal{X}_{f}$, there exists a $u$ such that $\mathbb{E}_{w}[f(x,u,w,\theta^{\ast})] \in \mathcal{X}_{f}$, and $\mathbb{E}_{w}[V_{f}(f(x,u,w,\theta^{\ast}))]-V_{f}(x) \le -l(x,u)$,
\end{itemize}
then the origin is asymptotically stable in $\mathcal{X}$ for the Bayesian risk-averse MPC-controlled system $x_{k+1}=f(x_{k},\kappa_{N}(x_{k},k),w_{k},\theta)$, w.p.1.
\end{theorem}

\begin{proof}
The proof can be found in Appendix \ref{appendix:Bayesian RA MPC stability}.
\end{proof}

\noindent \textbf{Remark.} Note that the Bayesian risk-averse MPC stability condition (Theorem \ref{thm:Bayesian RA MPC stability}) is more relaxed than the general risk-averse MPC stability condition (Theorem \ref{thm:RA MPC stability}), since the descent property of $V_{f}(\cdot)$ (condition (iii)) is now only imposed on the true system parameter $\theta^{\ast}$ rather than all the candidate parameters in the ambiguity set $\theta \in \mathcal{A}$.  This is due to the consistency guarantee of Bayesian learning (Theorem \ref{thm: consistency}). As long as the true system is identifiable and can be stabilized, the proposed algorithm is guaranteed to achieve stability eventually, w.p.1.

\subsection{Sub-optimal Control Policy}\label{subsec:sub-optimal policy}

When the radius of the ambiguity set $\mathcal{A}_{k}$ is relatively large, it is computationally intractable to find the optimal control policy even with the parameterized control policy. Actually, the optimal control policy does not necessarily demonstrate evident benefit over some sub-optimal control policies if there is significant uncertainty when making the control decision. Therefore, we consider an alternative sub-optimal control policy that requires little online optimization, but still ensures asymptotic stability of the time-varying closed-loop system.

Sub-optimal MPC considers the evolution of an extended state consisting of the system state and a warm-start of control policy. Given a feasible warm start, optimization algorithms can produce an improved feasible control policy, or even simply return the warm start. The first input is then injected and a new warm start can be generated from the returned control policy and a feasible terminal control law.

Therefore, the procedure of sub-optimal MPC is to repeat the following steps:

\noindent (i) Suppose that, at time $k$ and state $x_{k}$, a feasible parameterized control policy with a sequence of parameters $\boldsymbol{v}(k)$ has been determined, and so has the control action $\mu(x_{k},v_{0}(k))$. Then at the next time $k+1$, the subsequent state is $x_{k+1}=f(x_{k},\mu(x_{k},v_{0}(k)),w_{k},\theta^{\ast})$, with $w_{k}$ the value of the noise disturbance at time $k$. Consider a new control policy
\begin{equation}\label{eqn:warm start}
    \tilde{\boldsymbol{v}}(k+1):=\{\boldsymbol{v}_{1:N-1}(k), v_{f}\},
\end{equation}
where $v_{f}$ can be any feasible parameters such that the parameterized control policy $\mu(\cdot, v_{f})$ satisfies condition (iii) in Theorem \ref{thm:Bayesian RA MPC stability}.

Then $\tilde{\boldsymbol{v}}(k+1)$ is a feasible control policy for the subsequent state $x_{k+1}$, and thus can be utilized as a feasible warm-start of control policy at time $k+1$.

\noindent (ii) Next, a few optimization steps (e.g., value iteration, policy gradient, or evolutionary algorithms) can be executed to compute a better control sequence $\boldsymbol{v}(k+1)$ such that
\begin{equation}\label{eqn:sub-optimal policy}
    V_{N}(x_{k+1},\boldsymbol{\mu}(\boldsymbol{v}(k+1)),k+1) \le V_{N}(x_{k+1},\boldsymbol{\mu}(\tilde{\boldsymbol{v}}(k+1)),k+1).
\end{equation}

The sub-optimal control policy is thus given as $\mu_{0}(x_{k+1},k+1)=\mu(x_{k+1},v_{0}(k+1))$, with $v_{0}(k+1)$ the first element of the sub-optimal decision variable sequence $\boldsymbol{v}(k+1)$.
The sub-optimal control policy $\mu(x_{k},v_{0}(k))$ can also guarantee the asymptotic stability of the time-varying closed-loop system, if the Bayesian risk-averse MPC stability conditions in Theorem \ref{thm:Bayesian RA MPC stability} are satisfied.

\begin{theorem}[Stability of Sub-optimal MPC]\label{thm:sub-optimal stability}
If the Bayesian risk-averse MPC stability conditions (i-iii) in Theorem \ref{thm:Bayesian RA MPC stability} are satisfied, and 
\begin{itemize}
    \item[(iv)] There exists a $\mathcal{K}_{\infty}$ functions $\alpha_{l}(\cdot)$ satisfying $l(x,u) \ge \alpha_{l}(|(x,u)|)$, $\forall x \in \mathcal{X}, u \in \mathcal{U}$,
\end{itemize}
then the origin is asymptotically stable in $\mathcal{X}$ for the sub-optimal MPC-controlled system $x_{k+1}=f(x_{k},\mu(x_{k},v_{0}(k)),w_{k},\theta)$, w.p.1.
\end{theorem}

\begin{proof}
The proof can be found in Appendix \ref{appendix:sub-optimal stability}.
\end{proof}

\noindent \textbf{Remark 1.} Compared with Theorem \ref{thm:Bayesian RA MPC stability}, by imposing a mild additional condition, the requirement of optimal control policy is now unnecessary. The system can still be stabilized with a computationally efficient sub-optimal control policy. Actually, the additional requirement (condition (iv)) may be naturally satisfied if conditions (i-iii) in Theorem \ref{thm:Bayesian RA MPC stability} are satisfied. For example, in LQR, the choice of stage cost $l(x,u)=\frac{1}{2}(x^{\mathsf{T}}Qx+u^{\mathsf{T}}Ru)$ naturally satisfies condition (iv). 

\noindent \textbf{Remark 2.} 
Note that in this paper, the stability conditions are derived under the assumption that the Bayesian consistency conditions can be satisfied during the system trajectory. 
However, the control policy itself can influence the system trajectory, causing the consistency and stability conditions to become intertwined. 
Designing a controller that ensures both consistency and stability requires explicitly addressing the trade-off between exploration and exploitation (cf. \cite{Curi2020EfficientPlanning}). 
One promising direction for future research is to combine Theorem~\ref{thm:sub-optimal stability} and Theorem~\ref{thm: consistency} in an active inference \cite{Smith2022AData} setting to tackle this dual learning-control problem.

Now we are ready to introduce a practical Bayesian risk-averse MPC algorithm, which is presented in Algorithm~\ref{alg:BRAC}.

\begin{algorithm}[ht]
\SetAlgoLined
\caption{Bayesian Risk-Averse MPC}\label{alg:BRAC}
Initialize prior $\pi_{0}(\theta)$ with uniform probability distribution over the parameter space $\Theta$\\
Sample $N_{s}$ equally-weighted particles $\{\frac{1}{N_{s}},\theta_{0,i}^{+}\}$ from the prior $\pi_{0}$\\
Initialize state $x_{0}$, input $u_{0}=0$,  iterator $k=1$\\
\While{Not converge}{
Take an observation of $x_{k}$\\
Generate sample $e_{k-1,i}$ from a distribution of the small random noise $e_{k-1}$ for each particle $i$, then propagate particles $\theta_{k,i}^{-}=\theta_{k-1,i}^{+} + e_{k-1,i}$\\
Compute the weights $m_{k,i}=\frac{q(x_{k};\theta_{k,i}^{-},x_{k-1},u_{k-1})}{\sum_{j=1}^{N_{s}}q(x_{k};\theta_{k,j}^{-},x_{k-1},u_{k-1})}, i = 1, ... N_s$\\
Resample $\{m_{k,i},\theta_{k,i}^{-}\}$ to obtain $N_{s}$ new equally-weighted particles $\{\frac{1}{N_{s}},\theta_{k,i}^{+}\}$\\
Construct the ambiguity set $\mathcal{A}_{k}$ from the empirical distribution (particles) according to (\ref{eqn:ambiguity set})\\
Obtain a feasible warm start based on the control policy of the previous time according to (\ref{eqn:warm start})\\
Take several optimization steps to calculate a better control policy that satisfies (\ref{eqn:sub-optimal policy})\\
Execute the control action $u_{k}=\mu(x_{k},v_{0}(x_{k},k))$\\
$k := k + 1$\\
}
\end{algorithm}

\section{Simulation}\label{sec:simulation}

To evaluate the proposed Bayesian risk-averse MPC framework, we conduct simulations on a nonlinear cart-pole system with partially unknown dynamics. The objective is to swing up and stabilize the pole in the upright position while minimizing control effort and enforcing physical constraints.

\noindent \textbf{System Dynamics.}
The cart-pole system is described by the following continuous-time dynamics:
\begin{align*}
\dot{p} &= \dot{p}, \nonumber \\
\ddot{p} &= \frac{u + m \sin q (\ell \dot{q}^2 + g \cos q)}{M + m \sin^2 q}, \nonumber \\
\dot{q} &= \dot{q}, \nonumber \\
\ddot{q} &= \frac{-u \cos q - m \ell \dot{q}^2 \cos q \sin q - (M + m) g \sin q}{\ell (M + m \sin^2 q)},
\end{align*}
where $p$ and $q$ denote the cart position and the pole angle, respectively; $u$ is the control input; $M = 1.0$ kg is the known cart mass; $m$ and $\ell$ are the \emph{unknown} pole mass and pole length, respectively; and $g = 9.81$ m/s$^2$ is gravity. The dynamics are discretized using a fourth-order Runge--Kutta method with a time step of $\Delta t = 0.05$ s.

\noindent \textbf{Parameter Estimation.}
We treat the pole mass and length $\theta=[m,l]^{\intercal}$ as a latent variable and estimate it online using a particle filter with 1000 particles. The full state is assumed measurable with no observation noise, while Gaussian process noise with standard deviation $\sigma_w = 0.01$ is added to simulate unmodeled disturbances.
At each timestep, particles are propagated using the system dynamics and weighted based on their consistency with observed states. 

\noindent \textbf{Baselines.}
We compare our proposed risk-averse strategy (with credible level 0.9) against three classic MPC baselines:
\begin{itemize}
    \item Nominal MPC: Plans using the posterior mean.
    \item Tube MPC: Plans against the worst-case within a pre-defined tube of the unknown parameters.
    \item Stochastic MPC: Plans with a risk-neutral perspective using the posterior distribution.
\end{itemize}

\noindent \textbf{Implementations.}
Each controller uses a 5-step planning horizon in a receding horizon fashion, with only the first control input applied at each step. Optimization is solved using L-BFGS-B via \texttt{scipy.optimize.minimize}, warm-started from the previous control sequence.

Control input is bounded by: $u \in [-10, 10]$ N. The initial state for all runs is $[p, \dot{p}, q, \dot{q}] = [0, 0, 0, 0]$ (hanging down).

\noindent \textbf{Evaluation Metrics.}
Each controller is evaluated over 50 Monte Carlo simulations with different noise realizations and particle samples. We report:
\begin{itemize}
    \item Total Cost: Cumulative cost during 100 time steps. The stage cost is chosen as $c(x,u)=(x-x^{*})^{\intercal}Q(x-x^{*})+u^{\intercal}Ru$, where $Q=\diag(1, 0.1, 10, 0.1)$, $R=0.01$, and $x^{*}=[0, 0, \pi, 0]^{\intercal}$.
    \item Tracking Error: Mean-squared deviation of $q$ from $\pi$.
    \item Parameter Estimation Error: Deviation of estimated $\theta$ from true $\theta$.
    \item Angle Constraint Violations: Count on violations of angle constraint ($|q|\le \pi$) during 100 time steps. 
\end{itemize}

\noindent \textbf{Results and Discussions.}
Figure~\ref{fig:metrics} presents a comparative evaluation of Nominal MPC, Tube MPC, Stochastic MPC, and the proposed Risk-averse MPC over 50 Monte Carlo trials on the cart-pole system with an unknown pole mass and length. The proposed Risk-averse MPC consistently outperforms the baselines across multiple dimensions. It achieves the lowest total cost and tracking error, indicating efficient and accurate control under epistemic uncertainty. In contrast, Tube MPC exhibits the highest cost and worst tracking, attributed to its conservative design that prioritizes robustness at the expense of performance. While Nominal and Stochastic MPC perform comparably in terms of cost and tracking, their higher variance suggests inconsistent behavior across trials. Regarding parameter estimation, Tube MPC results in significantly larger estimation errors due to its overly conservative trajectories that avoid informative exploration. The other approaches—including Risk-averse MPC—enable accurate parameter learning. In terms of constraint satisfaction, Tube MPC yields the fewest angle constraint violations, but this comes at a high performance cost. Risk-averse MPC achieves a favorable balance, reducing constraint violations compared to Nominal and Stochastic MPC while maintaining superior control performance. Overall, these results demonstrate that Risk-averse MPC provides an effective compromise between safety, learning, and control efficiency in uncertain dynamical systems.

\begin{figure*}[t]
    \centering
    \includegraphics[width=0.95\linewidth]{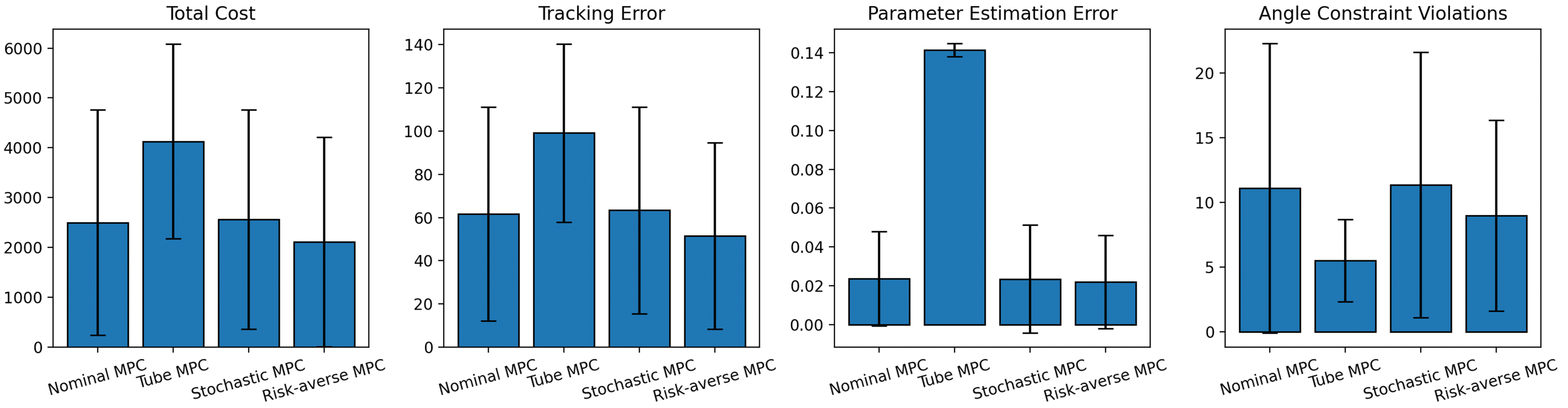}
    \caption{Performance comparison across Nominal, Tube, Stochastic, and the proposed Risk-verse MPC over 50 Monte Carlo trials on the cart-pole system with an unknown pole mass and length. Metrics include total cost, tracking error, estimation accuracy, angle constraint violations. Error bars denote one standard deviation. Our method consistently balances performance, robustness, and learning efficiency under epistemic uncertainty.}
    \label{fig:metrics}
\end{figure*}

\section{Conclusion and Future Work}\label{sec:conclusion}

In this paper, an online Bayesian learning-based risk-averse MPC framework was proposed to deal with the uncertainty that exists in online data-driven MPC from a risk-averse perspective. 
A practical Bayesian risk-averse MPC algorithm has been developed to capacitate computationally tractable implementation with theoretical guarantees of consistency and stability.
Future work includes formally incorporating constraints into the problem and theoretically analyzing the recursive feasibility of the proposed framework.
Another interesting and promising direction to explore is the dual learning-control problem and developing a principled approach to design a controller that ensures both consistency and stability.
One limitation of the current approach lies in its scalability to high-dimensional state and parameter spaces due to the curse of dimensionality. Addressing this challenge through techniques such as dimension reduction or scalable variational approximations remains an important direction for future work.


%

\appendices
\section{Proof of Theorem \ref{thm: consistency}(Bayesian Consistency Conditions)}\label{appendix:consistancy}

To prove the consistency of the Bayesian estimator, we first prove Lemma~\ref{lemma:likelihood convergence} based on condition (i).

\begin{lemma}\label{lemma:likelihood convergence}
The marginal transition kernel $\hat{q}(x_{k};x_{k-1},u_{k-1})
=\sum_{\theta}\pi_{k-1}(\theta)q(x_{k};\theta,x_{k-1},u_{k-1})$ converges to the true transition kernel $q(x_{k};\theta^{\ast},x_{k-1},u_{k-1})$, i.e., $\lim_{k\rightarrow \infty} \hat{q}(x_{k};x_{k-1},u_{k-1})=q(x_{k};\theta^{\ast},x_{k-1},u_{k-1})$, with probability 1.
\end{lemma}

\begin{proof}
Let $\mathcal{F}_k = \sigma\left\{(x_{s},u_{s}), s\leq k\right\}$ be the $\sigma-$filtration generated by the past observed states and controls.
According to \eqref{eqn:posterior distribution}, the estimated probability of the true parameter satisfies the following equation,
\begin{equation*}
    \log \pi_{k}(\theta^{\ast})=\log \pi_{k-1}(\theta^{\ast}) + \log \frac{q(x_{k};\theta^{\ast},x_{k-1},u_{k-1})}{\hat{q}(x_{k};x_{k-1},u_{k-1})}.
\end{equation*}

Taking expectation on both sides, we have that
\begin{align*}
    &\mathbb{E}[ \log \pi_{k}(\theta^{\ast})]  \\=& \mathbb{E}[\log \pi_{k-1}(\theta^{\ast})] + \mathbb{E}\left[ \log \frac{q(x_{k};\theta^{\ast},x_{k-1},u_{k-1})}{\hat{q}(x_{k};x_{k-1},u_{k-1})}\right]\\
    =& \mathbb{E}[\log \pi_{k-1}(\theta^{\ast})] \\
    & ~~+ \mathbb{E}\left[\mathbb{E}\left[ \log \frac{q(x_{k};\theta^{\ast},x_{k-1},u_{k-1})}{\hat{q}(x_{k};x_{k-1},u_{k-1})}|\mathcal{F}_{k-1}\right]\right]\\
    =& \mathbb{E}[\log \pi_{k-1}(\theta^{\ast})] +\mathbb{E}[d_{k-1}]
\end{align*}
where {\small $d_{k-1}=DL(q(x_{k};\theta^{\ast},x_{k-1},u_{k-1})||\hat{q}(x_{k};x_{k-1},u_{k-1}))$} is the \textit{relative entropy} (\textit{Kullback–Leibler divergence}) \cite{Klenke2014ProbabilityCourse} between $q(x_{k};\theta^{\ast},x_{k-1},u_{k-1})$ and $\hat{q}(x_{k};x_{k-1},u_{k-1})$. Then the expectation of $d_{k-1}$ can be represented as follows,
\begin{equation*}
    \mathbb{E}[d_{k-1}] = \mathbb{E}[\log \pi_{k}(\theta^{\ast})]- \mathbb{E}[\log \pi_{k-1}(\theta^{\ast})] .
\end{equation*}

For any $n,$ taking summation over $k$ from $1$ to $n$ on both sides, we have
\begin{equation*}
    \sum_{k=1}^{n}\mathbb{E}[d_{k-1}] = \mathbb{E}[\log \pi_{n}(\theta^{\ast})] - \log \pi_{0}(\theta^{\ast}) \le -\log \pi_{0}(\theta^{\ast}) < \infty,
\end{equation*}
where the last inequality holds according to condition (i).

Therefore, take $n\rightarrow \infty$ and we get
\begin{equation*}
    \sum_{k=1}^{\infty}\mathbb{E}[d_{k-1}] \le -\log \pi_{0}(\theta^{\ast}) < \infty.
\end{equation*}

By \textit{Markov Inequality}, we know that for any $\epsilon>0,$
$$\sum _{k=0}^\infty \mathrm{Pr}[d_k\geq \epsilon]\leq \frac{1}{\epsilon}\sum _{k=0}^\infty \mathbb{E}[d_k]<\infty.$$

We can then apply \textit{Borel-Cantelli Lemma} and show that $P(d_k\geq\epsilon, i.o.)=0,$ which further implies $\lim_{k\rightarrow \infty}d_k = 0,$ with probability 1.

Moreover, since $d_k\geq 0,$ by \textit{Tonelli's Theorem}, we have 
$$ \mathbb{E}\left[\sum _{k=0}^\infty d_k\right] = \sum _{k=0}^\infty \mathbb{E}[d_k] \leq -\log \pi_{0}(\theta^{\ast}).$$

Since $\sum _{k=0}^\infty d_k$ has bounded expectation, it must be finite with probability 1.

Note that the \textit{total variation distance} between two distributions is related to the \textit{relative entropy} by \textit{Pinsker's Inequality}: 
{\small
\begin{equation*}
    ||q(x_{k};\theta^{\ast},x_{k-1},u_{k-1})-\hat{q}(x_{k};x_{k-1},u_{k-1})||_{TV}\le \sqrt{2d_{k-1}},
\end{equation*}
}
where 
\begin{equation*}
\begin{split}
    &||q(x_{k};\theta^{\ast},x_{k-1},u_{k-1})-\hat{q}(x_{k};x_{k-1},u_{k-1})||_{TV}=\\
    &\sup_{x_{k}}|q(x_{k};\theta^{\ast},x_{k-1},u_{k-1})-\hat{q}(x_{k};x_{k-1},u_{k-1})|.
\end{split}
\end{equation*}

Letting $k\rightarrow \infty,$ by the convergence of $d_k,$ we have
{\small
\begin{equation*}
    \lim_{k\rightarrow \infty} \int_{x_{k}}|q(x_{k};\theta^{\ast},x_{k-1},u_{k-1})-\hat{q}(x_{k};x_{k-1},u_{k-1})|dx_{k} = 0,
\end{equation*}
}
with probability 1.

According to \textit{Dominated Convergence Theorem}, we further have 
{\small
\begin{equation*}
    \int_{x_{k}} \lim_{k\rightarrow \infty} |q(x_{k};\theta^{\ast},x_{k-1},u_{k-1})-\hat{q}(x_{k};x_{k-1},u_{k-1})|dx_{k} = 0.
\end{equation*}
}
Moreover, since $$|q(x_{k};\theta^{\ast},x_{k-1},u_{k-1})-\hat{q}(x_{k};x_{k-1},u_{k-1})| \ge 0$$ and $q(x_{k};\theta^{\ast},x_{k-1},u_{k-1})-\hat{q}(x_{k};x_{k-1},u_{k-1})$ is continuous in $x_{k}$, then for any $x_{k}$,
\begin{equation*}
    \lim_{k\rightarrow \infty}|q(x_{k};\theta^{\ast},x_{k-1},u_{k-1})-\hat{q}(x_{k};x_{k-1},u_{k-1})|=0,
\end{equation*}
which means
\begin{equation}\label{eqn:measurement convergence}
    \lim_{k\rightarrow \infty} \hat{q}(x_{k};x_{k-1},u_{k-1})=q(x_{k};\theta^{\ast},x_{k-1},u_{k-1}),
\end{equation}
with probability 1.
\end{proof}

Then we prove that the posterior distribution of $\theta$ converges to the distribution $\delta_{\theta^{\ast}}(\theta)$ based on conditions (ii) and (iii).

\begin{proof}
Note that
\begin{equation}\label{eqn:measurement distance}
\begin{split}
    &q(x_{k};\theta^{\ast},x_{k-1},u_{k-1})-\hat{q}(x_{k};x_{k-1},u_{k-1}) \\
    =&[1-\pi_{k-1}(\theta^{\ast})]q(x_{k};\theta^{\ast},x_{k-1},u_{k-1})\\
    &-\sum_{\theta \ne \theta^{\ast}}\pi_{k-1}(\theta)q(x_{k};\theta,x_{k-1},u_{k-1}).
\end{split}
\end{equation}

Note that for any $t > 0$, $(\pi_{t}(\theta_1), \pi_{t}(\theta_2), \cdots)$ is an infinitely dimensional bounded vector with all components sum up to 1, we can take a subsequence $\{\pi_{t_k}\}$ such that for each component $j$, $\{\pi_{t_k}(\theta_j)\}$ converges to a limit which is denoted by $\pi_{\infty}(\theta_j)$, which is also known as weak convergence (of a deterministic sequence). 

Next, we will show that $\pi_{\infty}(\theta)$ is a normalized vector. Note that for any $j \in \mathbb{N}$, $\lim_{t_k \to \infty} \pi_{t_k}(\theta_j) = \pi_{\infty}(\theta_j)$, which is equivalent to
\begin{equation*}
    \forall \epsilon_j > 0, \exists N \in \mathbb{N}, s.t. \forall n \geq N, |\pi_{\infty}(\theta_j) - \pi_{n}(\theta_j)| \leq \epsilon.
\end{equation*}
Therefore, we have
\begin{equation}\label{eqn: posterior_sum}
    - \epsilon_j < \pi_{\infty}(\theta_j) - \pi_n(\theta_j) < \epsilon_j, j=1,2,\cdots
\end{equation}

According to the Bayesian update rule, we know $\sum_{j=1}^{\infty}\pi_{n}(\theta_j) = 1$. It then follows that $\forall \epsilon > 0$, take $\epsilon_j = \frac{\epsilon}{2^{j}}$ and sum over \eqref{eqn: posterior_sum} for
all $j \in \mathbb{N}$, we get
\begin{equation*}
    -(\frac{\epsilon}{2^{1}} + \frac{\epsilon}{2^{2}} + \cdots) <  \sum_{j=1}^{\infty} \pi_{\infty}(\theta_j) - 1 < (\frac{\epsilon}{2^{1}} + \frac{\epsilon}{2^{2}} + \cdots),
\end{equation*}
which indicates $\forall \epsilon > 0$, $|\sum_{j=1}^{\infty} \pi_{\infty}(\theta_j)-1| < \epsilon$, and it implies that $\sum_{j=1}^{\infty} \pi_{\infty}(\theta_j) = 1$. So the limit is also a valid probability simplex. 

Since every weakly convergent sequence in $L^{1}$ is strongly convergent (cf. Chapter 2 in \cite{Pedersen2012AnalysisNow}), then any convergent subsequence of $\{\pi_{t_k}\}$ has the same limit $(\pi_{\infty}(\theta_1), \pi_{\infty}(\theta_2), \cdots)$.

According to condition (ii), take limit over \eqref{eqn:measurement distance} along each convergent subsequence $\{x^{m}_{k},u^{m}_{k}\}_{k=1:\infty}$, and by \eqref{eqn:measurement convergence}, we have $\forall m \in \mathcal{M}$,
{\small
\begin{equation}\label{eqn:strongly distinguishable}
    [1-\pi_{\infty}(\theta^{\ast})]q(\cdot;\theta^{\ast},x^{m}_{\infty},u^{m}_{\infty})-\sum_{\theta \ne \theta^{\ast}}\pi_{\infty}(\theta)q(\cdot;\theta,x^{m}_{\infty},u^{m}_{\infty}) = 0, 
\end{equation}
}
with probability 1.

Since the combination of blind regions of $\mathcal{P}(\eta_{m})$ in the converged context $\eta_{m}=[x^{m}_{\infty},u^{m}_{\infty}]$ is empty, then from Proposition \ref{prop:comb BR}, the blind region of $\{\mathcal{P}(\eta_{m}), m \in \mathcal{M}\}$ is empty.
Therefore, the elements within $\Theta$ are strongly observationally distinguishable from $\{\mathcal{P}(\eta_{m}), m \in \mathcal{M}\}$. Then from \eqref{eqn:strongly distinguishable} we have
$
1-\pi_{\infty}(\theta^{\ast})=0, 
\pi_{\infty}(\theta)=0 ~~\forall \theta \ne \theta^{\ast},
$
which implies
\begin{equation*}
    \lim_{k\rightarrow \infty} \pi_{k}(\theta)=\delta_{\theta^{\ast}}(\theta),
\end{equation*}
with probability 1.
\end{proof}

\section{Proof of Theorem \ref{thm:RA Lyapunov stability} (Risk-averse Lyapunov Stability Theorem)}\label{appendix:RA Lyapunov stability}

To prove Theorem \ref{thm:RA Lyapunov stability}, we make use of some properties of the comparison function classes, whose proofs can be found in \cite{Kellett2014AResults} and \cite{Bof2018LyapunovSystems}.

\begin{lemma}
If $\alpha_{1}$ and $\alpha_{2}$ are $\mathcal{K}$ functions ($\mathcal{K}_{\infty}$ functions), then $\alpha_{1}^{-1}$ and $(\alpha_{1}\circ \alpha_{2}):=\alpha_{1}(\alpha_{2}(\cdot))$ are $\mathcal{K}$ functions ($\mathcal{K}_{\infty}$ functions).
\end{lemma}

\begin{lemma}
If $\alpha$ is a $\mathcal{K}$ function and $\sigma$ is a $\mathcal{L}$ function, then $(\alpha \circ \sigma):=\alpha(\sigma(\cdot))$ is a $\mathcal{L}$ function.
\end{lemma}

\begin{lemma}
If $\alpha_{1}$ and $\alpha_{2}$ are $\mathcal{K}$ functions and $\beta$ is a $\mathcal{KL}$ function, then $\beta'(r,s):=\alpha_{1}(\beta(\alpha_{2}(r),s))$ is a $\mathcal{KL}$ function.
\end{lemma}

\begin{lemma}\label{lemma:rho lower bound}
Let $\rho: \mathbb{R}^{n} \rightarrow \mathbb{R}_{\ge 0}$ be a continuous positive definite function. Then, there exist functions $\alpha \in \mathcal{K}_{\infty}$ and $\sigma \in \mathcal{L}$ such that 
\begin{equation*}
    \rho(x) \ge \alpha(|x|)\sigma(|x|), \quad \forall x \in \mathbb{R}^{n}.
\end{equation*}
\end{lemma}

\begin{lemma}\label{lemma:rho bound}
Let $\rho: D \rightarrow \mathbb{R}_{\ge 0}$ be a continuous positive definite function defined on a domain $D \subset \mathbb{R}^{n}$ that contains the origin. Let $B_{r}:= \{x \in \mathbb{R}^{n_{x}}| |x| < r\} \subset D$ for some $r>0$. Then, there exist functions $\alpha_{1}, \alpha_{2} \in \mathcal{K}$ defined on $[0,r)$ such that
\begin{equation*}
    \alpha_{1}(|x|) \le \rho(x) \le \alpha_{2}(|x|), \quad \forall x \in B_{r}.
\end{equation*}
\end{lemma}

Now we can prove Theorem \ref{thm:RA Lyapunov stability} based on the above lemmas.

\begin{proof}
According to Lemma \ref{lemma:rho lower bound}, there exist functions $\alpha \in \mathcal{K}_{\infty}$ and $\sigma \in \mathcal{L}$ such that 
\begin{equation*}
    \rho(x) \ge \alpha(|x|)\sigma(|x|), \quad \forall x \in \mathbb{R}^{n_{x}}.
\end{equation*}

Since the functions $\alpha_{1},\alpha_{2} \in \mathcal{K}_{\infty}$, they are both invertible and we define 
\begin{equation*}
    \hat{\alpha} := \alpha \circ \alpha_{2}^{-1} \in \mathcal{K}_{\infty} \quad \text{and} \quad
    \hat{\sigma} := \sigma \circ \alpha_{1}^{-1} \in \mathcal{L}.
\end{equation*}

Finally, we define $\hat{\rho}: \mathbb{R}_{\ge 0} \rightarrow \mathbb{R}_{\ge 0}$ by
\begin{equation*}
    \hat{\rho}(s) := \hat{\alpha}(s)\hat{\sigma}(s), \quad \forall s \in \mathbb{R}_{\ge 0}.
\end{equation*}
Note that the product of a class-$\mathcal{K}_{\infty}$ function and a class-$\mathcal{L}$ function is a positive definite function and hence $\hat{\rho}$ is a positive definite function. 

According to condition (ii), we have
\begin{equation*}
    \begin{split}
        &\gamma[V(f(x,u,w,\theta))] - V(x) \le -\rho(x) 
        \le -\alpha(|x|)\sigma(|x|)  \\
        &\le -\alpha(\alpha_{2}^{-1}V(x))\sigma(\alpha_{1}^{-1}V(x)) = -\hat{\rho}(V(x)).
    \end{split}
\end{equation*}
If we further define 
$$
\bar{\rho}(s) := \inf \{ s, \hat{\rho}(s)\}, \quad \forall s \in \mathbb{R}_{\ge 0},
$$
then we also have 
\begin{equation}\label{eqn:decrease bound}
    \gamma[V(f(x,u,w,\theta))] - V(x) \le -\bar{\rho}(V(x)).
\end{equation}

According to the fact that
$
\sup_{\theta}\{a(\theta)\}-\sup_{\theta}\{b(\theta)\} \le \sup_{\theta}\{a(\theta)-b(\theta)\}
$, we have 
\begin{equation*}
\begin{split}
    &\gamma[\gamma[V(f(x,u,w,\theta))]] - \gamma[V(x)] \\
    &\le \gamma[\gamma[V(f(x,u,w,\theta))] - V(x)] 
    \le -\gamma[\bar{\rho}(V(x))] 
\end{split}
\end{equation*}

Let $x_{i}:=\phi(i;x,\theta,\boldsymbol{\mu},\boldsymbol{w})$, then by induction we have
\begin{equation*}
\begin{split}
    \bar{\gamma}_{i}[V(x_{i+1})] - \bar{\gamma}_{i-1}[V(x_{i})] & \le -\bar{\gamma}_{i-1}[\bar{\rho}(V(x_{i}))]. \\ 
\end{split}
\end{equation*}

Therefore, we have
\begin{equation*}
    \begin{split}
        \bar{\gamma}_{i-1}[V(x_{i})] &= V(x_{0}) + \sum_{k=0}^{i-1}[\bar{\gamma}_{k}[V(x_{k+1})] - \bar{\gamma}_{k-1}[V(x_{k})]] \\
        &\le V(x_{0}) - \sum_{k=0}^{i-1}\bar{\gamma}_{k-1}[\bar{\rho}(V(x_{k}))] \\
        &:= \bar{\beta}(V(x_{0}), i-1).
    \end{split}
\end{equation*}

Since $\bar{\rho}(\cdot)$ is positive definite, according to Lemma \ref{lemma:rho bound}, there exists a function $\alpha_{3} \in \mathcal{K}$ defined on $[0, V(x_{0}))$ such that 
$$
\bar{\rho}(s) \ge \alpha_{3}(|s|), \quad \forall s \in [0, V(x_{0})).
$$
Therefore, we have
\begin{equation*}
    \begin{split}
        \bar{\beta}(V(x_{0}), i) - \bar{\beta}(V(x_{0}), i-1) &= -\bar{\gamma}_{i-1}[\bar{\rho}(V(x_{i}))] \\
        &\le -\bar{\gamma}_{i-1}[\alpha_{3}(V(x_{i}))]
    \end{split}
\end{equation*}
which implies that $\bar{\beta}(V(x_{0}), \cdot)$ is strictly decreasing when $V(x_{i})>0$. Since $\bar{\beta}(V(x_{0}), i-1) \ge \bar{\gamma}_{i-1}[V(x_{i})]\ge 0$, we know that $\bar{\beta}(V(x_{0}), \cdot)$ is lower bounded by 0. Thus $\bar{\beta}(V(x_{0}), \cdot)$ is strictly decreasing to zero, i.e., $\bar{\beta}(V(x_{0}), \cdot) \in \mathcal{L}$.

Besides, define 
$$
\tilde{\rho}(s):= s - \bar{\rho}(s), \quad \forall s \in \mathbb{R}_{\ge 0},
$$
then $\tilde{\rho}$ is also a positive definite function. According to Lemma \ref{lemma:rho bound}, there exists a function $\alpha_{4} \in \mathcal{K}$ defined on $[0, V(x_{0}))$ such that 
$$
\tilde{\rho}(s) \le \alpha_{4}(|s|), \quad \forall s \in [0, V(x_{0})).
$$
Then we have
\begin{equation*}
    \begin{split}
        \bar{\beta}(V(x_{0}), i) \le \bar{\beta}(V(x_{0}), 0) = \tilde{\rho}(V(x_{0})) \le \alpha_{4}(V(x_{0})).
    \end{split}
\end{equation*}

Therefore, based on the analysis of $\bar{\beta}(V(x_{0}),i)$ above, there exists a function $\hat{\beta} \in \mathcal{KL}$ such that
$$
\bar{\gamma}_{i-1}[V(x_{i})] \le \bar{\beta}(V(x_{0}), i) \le \hat{\beta}(V(x_{0}), i), \quad \forall i \in \mathbb{N}.
$$

According to the $\textit{Monotonicity}$ of coherent risk measures, by manipulating the upper and lower bounds in condition (i) we obtain
$$
\bar{\gamma}_{i-1}[|x_{i}|] \le \alpha_{1}^{-1}\hat{\beta}(\alpha_{2}(x_{0}), i):=\beta(|x_{0}|,i),
$$
for all $i \in \mathbb{N}$, proving global asymptotic stability of the origin.
\end{proof}

\section{Proof of Theorem \ref{thm:RA MPC stability}(Risk-averse MPC Stability Conditions)}\label{appendix:RA MPC stability}

To prove the stability of the risk-averse MPC, we first prove the monotonicity property of the value function $V_{N}^{\ast}(x)$.

\begin{lemma}[Monotonicity of the value function]\label{lemma:Mono value}
The value function $V_{N}^{\ast}(x)$ is monotone, i.e.,
\begin{equation}
    V_{i+1}^{\ast}(x) \le V_{i}^{\ast}(x), \quad \forall i \in \mathbb{I}_{\ge 0}.
\end{equation}
\end{lemma}

\begin{proof}
According to condition (iii), there exists a $u$ satisfying 
$$
l(x,u) + \gamma [V_{f}(f(x,u,w,\theta))] \le V_{f}(x).
$$
Then for $i=0$, from the DP recursion (\ref{eqn:DP value})
\begin{equation*}
    \begin{split}
        V_{1}^{\ast}(x) &= \inf_{u} l(x,u) + \gamma [V_{0}^{\ast}(f(x,u,w,\theta))] \\
        &= \inf_{u} l(x,u) + \gamma [V_{f}(f(x,u,w,\theta))]  \\
        &\le V_{f}(x) = V_{0}^{\ast}(x).
    \end{split}
\end{equation*}

Next, suppose that for some $i \ge 1$,
$
V_{i}^{\ast}(x) \le V_{i-1}^{\ast}(x).
$
Then, using the DP recursion (\ref{eqn:DP value})
\begin{equation*}
    \begin{split}
        &V_{i+1}^{\ast}(x) - V_{i}^{\ast}(x) \\ = &l(x,\kappa_{i+1}(x)) + \gamma[V_{i}^{\ast}(f(x,\kappa_{i+1}(x),w,\theta))] \\
        &-l(x,\kappa_{i}(x)) - \gamma[V_{i-1}^{\ast}(f(x,\kappa_{i}(x),w,\theta))] \\
        \le &l(x,\kappa_{i}(x)) + \gamma[V_{i}^{\ast}(f(x,\kappa_{i}(x),w,\theta))] \\
        &-l(x,\kappa_{i}(x)) - \gamma[V_{i-1}^{\ast}(f(x,\kappa_{i}(x),w,\theta))] 
    \end{split}
\end{equation*}
since $\kappa_{i}(x)$ may not be optimal as $\kappa_{i+1}(x)$.

According to the fact that
$
\sup_{\theta}\{a(\theta)\}-\sup_{\theta}\{b(\theta)\} \le \sup_{\theta}\{a(\theta)-b(\theta)\}
$, we have
\begin{equation*}
    \begin{split}
        &\gamma[V_{i}^{\ast}(f(x,\kappa_{i}(x),w,\theta))] -\gamma[V_{i-1}^{\ast}(f(x,\kappa_{i}(x),w,\theta))] \\
        = &\sup_{\theta \in \mathcal{A}}\mathbb{E}_{w}[V_{i}^{\ast}(f(x,\kappa_{i}(x),w,\theta))] \\
        &-\sup_{\theta \in \mathcal{A}}\mathbb{E}_{w}[V_{i-1}^{\ast}(f(x,\kappa_{i}(x),w,\theta))] \\
        \le &\sup_{\theta \in \mathcal{A}}\mathbb{E}_{w}[V_{i}^{\ast}(f(x,\kappa_{i}(x),w,\theta)) - V_{i-1}^{\ast}(f(x,\kappa_{i}(x),w,\theta))]
    \end{split}
\end{equation*}

Since we assume that $V_{i}^{\ast}(x) \le V_{i-1}^{\ast}(x)$, we have 
$
V_{i+1}^{\ast}(x) - V_{i}^{\ast}(x) \le 0,
$
which implies that 
$
V_{i+1}^{\ast}(x) \le V_{i}^{\ast}(x).
$

By induction, we have
$
V_{i+1}^{\ast}(x) \le V_{i}^{\ast}(x), \quad \forall i \in \mathbb{I}_{\ge 0}.
$
\end{proof}

A standard approach to establish stability is to employ the value function as a valid Lyapunov function. Therefore we show below that $V_{N}^{\ast}(x)$ is a valid Lyapnov function for the MPC-controlled system $x_{k+1}=f(x_{k},\kappa_{N}(x_{k}),w_{k},\theta)$.

\begin{proof}
\noindent \textbf{Lower Bound for $V_{N}^{\ast}(x)$.}\\
$
V_{N}^{\ast}(x) \ge l(x, \kappa_{N}(x)) \ge \alpha_{1}(|x|).
$

\noindent \textbf{Upper Bound for $V_{N}^{\ast}(x)$.}
It resorts to the weak controllability assumption that
$
V_{N}^{\ast}(x) \le \alpha_{2}(|x|).
$

\noindent \textbf{Descent Property for $V_{N}^{\ast}(x)$.}
Note that
\begin{equation*}
\begin{split}
    V_{N}^{\ast}(x) = l(x, \kappa_{N}(x)) + \gamma[V_{N-1}^{\ast}(f(x,\kappa_{N}(x),w,\theta))].
\end{split}
\end{equation*}
According to the monotonicity property, Lemma \ref{lemma:Mono value}, we have
\begin{equation*}
    \begin{split}
        &\gamma[V_{N}^{\ast}(f(x,\kappa_{N}(x),w,\theta))] - V_{N}^{\ast}(x) \\
        = & \gamma[V_{N}^{\ast}(f(x,\kappa_{N}(x),w,\theta))]  - l(x, \kappa_{N}(x)) \\
        &- \gamma[V_{N-1}^{\ast}(f(x,\kappa_{N}(x),w,\theta))] \\
        \le & - l(x, \kappa_{N}(x)) \\
        &+ \gamma[V_{N}^{\ast}(f(x,\kappa_{N}(x),w,\theta)) - V_{N-1}^{\ast}(f(x,\kappa_{N}(x),w,\theta))] \\
        \le & - l(x, \kappa_{N}(x)) \le -\alpha_{1}(|x|).
    \end{split}
\end{equation*}

Therefore, $V_{N}^{\ast}(x)$ is a valid Lyapnov function for the risk-averse MPC-controlled system $x_{k+1}=f(x_{k},\kappa_{N}(x_{k}),w_{k},\theta)$. According to Theorem \ref{thm:RA Lyapunov stability}, the origin is RAAS in $\mathcal{X}$ for the risk-averse MPC-controlled system $x_{k+1}=f(x_{k},\kappa_{N}(x_{k}),w_{k},\theta)$.
\end{proof}

\section{Proof of Theorem \ref{thm:Bayesian RA MPC stability}(Bayesian Risk-averse MPC Stability Conditions)}\label{appendix:Bayesian RA MPC stability}

To prove the stability of the Bayesian risk-averse MPC, we first prove the relaxed monotonicity property of the time-varying value function $V_{N}^{\ast}(x,k)$.

Define 
\begin{equation}\label{eq:delta}
    \delta_{k} := \sup_{\theta \in \mathcal{A}_{k}}\{\mathbb{E}_{w}[V_{f}(f(x,u,w,\theta))]-V_{f}(x) + l(x,u)\}
\end{equation}

\begin{lemma}[Relaxed Monotonicity of the value function]\label{lemma:Relaxed Mono value}
The value function $V_{N}^{\ast}(x,k)$ is monotone with a relaxed bound $\delta_{k}$, i.e.,
\begin{equation}
    V_{i+1}^{\ast}(x,k) \le V_{i}^{\ast}(x,k) + \delta_{k}, \quad \forall i \in \mathbb{I}_{\ge 0}.
\end{equation}
\end{lemma}

\begin{proof}
According to (\ref{eq:delta}), there exists a $u$ satisfying 
$$
l(x,u) + \mathbb{E}_{w} [V_{f}(f(x,u,w,\theta))] \le V_{f}(x) + \delta_{k}, \quad \forall \theta \in \mathcal{A}_{k}.
$$
Then, following a similar induction procedure as in Lemma \ref{lemma:Mono value}, we can obtain
$
V_{i+1}^{\ast}(x,k) \le V_{i}^{\ast}(x,k) + \delta_{k}, \quad \forall i \in \mathbb{I}_{\ge 0}.
$
\end{proof}

Since the control policy $\kappa_{N}(x,k)$ is time-varying, the controlled closed-loop system is also time-varying. Therefore to establish stability for the time-varying system, we show below that the time-varying value function $V_{N}^{\ast}(x,k)$ can become a valid Lyapunov function for the time-varying MPC-controlled system $x_{k+1}=f(x_{k},\kappa_{N}(x_{k},k),w_{k},\theta)$, when $k$ is sufficiently large, i.e., given enough time.

\begin{proof}
\noindent \textbf{Lower Bound for $V_{N}^{\ast}(x,k)$.}\\
$
V_{N}^{\ast}(x,k) \ge l(x, \kappa_{N}(x,k)) \ge \alpha_{1}(|x|)
$

\noindent \textbf{Upper Bound for $V_{N}^{\ast}(x,k)$.}
$
V_{N}^{\ast}(x,k) \le \alpha_{2}(|x|).
$

\noindent \textbf{Descent Property for $V_{N}^{\ast}(x,k)$.}
Since the ambiguity sets satisfy
$\mathcal{A}_{k+1} \subseteq \mathcal{A}_{k},$
we have
\begin{equation*}
    \begin{split}
        V_{N}^{\ast}(x,k+1) &=\inf_{\boldsymbol{\mu}} \sup_{\theta \in \mathcal{A}_{k+1}} V_{N}(x, \theta, \boldsymbol{\mu}) \\
        &\le \inf_{\boldsymbol{\mu}} \sup_{\theta \in \mathcal{A}_{k}} V_{N}(x, \theta, \boldsymbol{\mu}) = V_{N}^{\ast}(x,k) 
    \end{split}
\end{equation*}

Note that
\begin{equation*}
\begin{split}
    V_{N}^{\ast}(x,k) &= l(x, \kappa_{N}(x,k)) \\
    &+ \sup_{\theta \in \mathcal{A}_{k}}\mathbb{E}_{w}[V_{N-1}^{\ast}(f(x,\kappa_{N}(x,k),w,\theta),k)].
\end{split}
\end{equation*}
According to the relaxed monotonicity property, Lemma \ref{lemma:Relaxed Mono value}, we have
\begin{equation*}
    \begin{split}
        &\sup_{\theta \in \mathcal{A}_{k}}\mathbb{E}_{w}[V_{N}^{\ast}(f(x,\kappa_{N}(x,k),w,\theta),k+1)] - V_{N}^{\ast}(x,k) \\
        \le &\sup_{\theta \in \mathcal{A}_{k}}\mathbb{E}_{w}[V_{N}^{\ast}(f(x,\kappa_{N}(x,k),w,\theta),k)] - V_{N}^{\ast}(x,k) \\
        = & \sup_{\theta \in \mathcal{A}_{k}}\mathbb{E}_{w}[V_{N}^{\ast}(f(x,\kappa_{N}(x,k),w,\theta),k)] - l(x, \kappa_{N}(x,k)) \\
        &- \sup_{\theta \in \mathcal{A}_{k}}\mathbb{E}_{w}[V_{N-1}^{\ast}(f(x,\kappa_{N}(x,k),w,\theta),k)] \\
        \le & - l(x, \kappa_{N}(x,k)) + \sup_{\theta \in \mathcal{A}_{k}}\mathbb{E}_{w}[V_{N}^{\ast}(f(x,\kappa_{N}(x,k),w,\theta),k) \\
        &- V_{N-1}^{\ast}(f(x,\kappa_{N}(x,k),w,\theta),k)] \\
        \le & - l(x, \kappa_{N}(x,k)) + \delta_{k}.
    \end{split}
\end{equation*}

If the Bayesian estimator is consistent, then by the credible interval construction, the ambiguity set $\mathcal{A}_{k}$ will converge to a singleton that only contains $\theta^{\ast}$, i.e. $\mathcal{A}_{\infty}=\{\theta^{\ast}\}$, w.p.1. 

According to condition (iii), we have
\begin{equation*}
    \begin{split}
        \delta_{\infty}=\mathbb{E}_{w}[V_{f}(f(x,u,w,\theta^{\ast}))]-V_{f}(x) + l(x,u) \le 0.
    \end{split}
\end{equation*}

Then w.p.1., there exists $K>0$, such that when $k>K$, the descent property for $V_{N}^{\ast}(\cdot,k)$ can be established as
\begin{equation*}
\begin{split}
    &\mathbb{E}_{w}[V_{N}^{\ast}(f(x,\kappa_{N}(x,k),w,\theta^{\ast}),k+1)] - V_{N}^{\ast}(x,k) \\
    &\le - l(x, \kappa_{N}(x,k)) \le - \alpha_{1}(|x|).
\end{split}
\end{equation*}

Therefore, the time-varying value function $V_{N}^{\ast}(x,k)$ becomes a valid Lyapunov function for the time-varying MPC-controlled system $x_{k+1}=f(x_{k},\kappa_{N}(x_{k},k),w_{k},\theta)$, when $k$ is sufficiently large, i.e., given enough time. According to the Lyapunov theorem of stochastic systems \cite{Chatterjee2015OnTechniques}, the origin is asymptotically stable in $\mathcal{X}$ for the Bayesian risk-averse MPC-controlled system $x_{k+1}=f(x_{k},\kappa_{N}(x_{k},k),w_{k},\theta)$.
\end{proof}

\section{Proof of Theorem \ref{thm:sub-optimal stability}(Stability of Sub-optimal MPC)}\label{appendix:sub-optimal stability}

In the sub-optimal MPC algorithm, denote the extended state, consisting of the state and warm-start pair, as $z:=(x,\boldsymbol{\mu}(\tilde{\boldsymbol{v}}))$. Then the extended state evolves according to 
\begin{equation}\label{eqn:extended state}
\begin{split}
    z^{+} \in H(z) := \{ (x^{+}, \boldsymbol{\mu}(\tilde{\boldsymbol{v}})^{+}) | &x^{+} = f(x,\mu(x,v_{0}),w,
    \theta^{\ast}), \\
    &\boldsymbol{\mu}(\tilde{\boldsymbol{v}})^{+} = \Psi(x,\boldsymbol{\mu}(\tilde{\boldsymbol{v}})) \}
\end{split}
\end{equation}
where $\boldsymbol{\mu}(\tilde{\boldsymbol{v}})^{+} = \Psi(x,\boldsymbol{\mu}(\tilde{\boldsymbol{v}}))$ denotes the mapping from the warm-start of the current step to the next step. 

To prove the asymptotic stability w.r.t. state $x$, we first prove the asymptotic stability w.r.t. the extended state $z$.
In order to directly link the asymptotic behavior of $z$ with that of $x$, the following proposition is necessary.

\begin{proposition}[Linking warm-start and state]\label{prop: link pair}
There exists a function $\alpha_{r}(\cdot) \in \mathcal{K}_{\infty}$ such that $|\boldsymbol{\mu}(\tilde{\boldsymbol{v}})| \le \alpha_{r}(|x|)$ for any feasible $x$ and $\tilde{\boldsymbol{v}}$.
\end{proposition}
\begin{proof}
The proof can be found in Proposition 10 of \cite{Allan2017OnMPC}.
\end{proof}

Since the extended state evolves as the difference inclusion, we utilize the following Lyapunov stability theorem for difference inclusion systems.

\begin{theorem}[Lyapunov Stability Theorem (Difference Inclusion)]\label{thm:Lyapunov Stability (Difference Inclusion)}
If the set $\mathcal{Z}$ contains the origin, is positive invariant for the difference inclusion $z^{+} \in H(z)$, $H(0)=0$, and there exists a Lyapunov function $V(\cdot)$ in $\mathcal{Z}$ such that for all $z \in \mathcal{Z}$:
\begin{itemize}
    \item[(i)] $\alpha_{1}(|z|) \le V(z) \le \alpha_{2}(|z|)$, where $\alpha_{1},\alpha_{2} \in \mathcal{K}_{\infty}$,
    \item[(ii)] $\sup_{z^{+}\in H(z)} V(z^{+}) \le V(z) - \alpha_{3}(|z|)$, where $\alpha_{3} \in \mathcal{K}_{\infty}$,
\end{itemize}
then the origin is asymptotically stable in $\mathcal{Z}$.
\end{theorem}
\begin{proof}
The proof can be found in Proposition 13 of \cite{Allan2017OnMPC}.
\end{proof}
 
According to Theorem \ref{thm:Lyapunov Stability (Difference Inclusion)}, to establish stability for the extended state $z$, we need to show that $V_{N}(z,k)$ is a valid Lyapunov function for the difference inclusion system \eqref{eqn:extended state}.

\begin{proof}
\noindent \textbf{Lower Bound for $V_{N}(z,k)$.}
\begin{equation*}
\begin{split}
V_{N}(z,k) &= V_{N}(x,\boldsymbol{\mu}(\tilde{\boldsymbol{v}}),k) \ge V_{N}(x,\boldsymbol{\mu}(\boldsymbol{v}),k) \\ 
&\ge \sum_{i=0}^{N-1} l(x_{i}, \mu(x_{i},v_{i})) \ge \sum_{i=0}^{N-1} \alpha_{l}(|(x_{i},\mu(x_{i},v_{i}))|)
\end{split}
\end{equation*}

According to the property of $\mathcal{K}$ functions that for $\alpha \in \mathcal{K}$ and all $c_{i}\in \mathbb{R}_{\ge 0},i \in \mathbb{I}_{1:n}$, $\alpha(c_{1}+\cdots+c_{n}) \le \alpha(n c_{1})+\cdots+\alpha(n c_{n})$, we have
\begin{equation*}
\begin{split}
\sum_{i=0}^{N-1} \alpha_{l}(|(x_{i},\mu(x_{i},v_{i}))|)
\ge \alpha_{l}(\frac{1}{N}\sum_{i=0}^{N-1}|(x_{i},\mu(x_{i},v_{i}))|) 
\end{split}
\end{equation*}

Then from the triangle inequality, we obtain
\begin{equation*}
\alpha_{l}(\frac{1}{N}\sum_{i=0}^{N-1}|(x_{i},\mu(x_{i},v_{i}))|) 
\ge \alpha_{l}(|(\boldsymbol{x},\boldsymbol{\mu}(\boldsymbol{v}))|/N)
\end{equation*}

Finally utilizing the $l_{p}$-norm property that for all vectors $\boldsymbol{a}, \boldsymbol{b}$, $|(\boldsymbol{a}, \boldsymbol{b})| \ge |\boldsymbol{b}|$, and noting that $x_{0}=x$, we have
\begin{equation*}
\begin{split}
\alpha_{l}(|(\boldsymbol{x},\boldsymbol{\mu}(\boldsymbol{v}))|/N) 
&\ge \alpha_{l}(|(x,\boldsymbol{\mu}(\boldsymbol{v}))|/N) \\
&:= \alpha_{1}(|(x,\boldsymbol{\mu}(\boldsymbol{v}))|)=\alpha_{1}(z)
\end{split}
\end{equation*}
where $\alpha_{1} \in \mathcal{K}_{\infty}$. Thus we have established the lower bound $V_{N}(z,k) \ge \alpha_{1}(z)$.

\noindent \textbf{Upper Bound for $V_{N}(z,k)$.}
$
V_{N}(z,k) \le \alpha_{2}(|z|).
$

\noindent \textbf{Descent Property for $V_{N}(z,k)$.}
Since the ambiguity sets satisfy
$\mathcal{A}_{k+1} \subseteq \mathcal{A}_{k},$
we have
\begin{equation*}
    \begin{split}
        V_{N}(z,k+1) &= \sup_{\theta \in \mathcal{A}_{k+1}} V_{N}(x, \theta, \boldsymbol{\mu}(\tilde{\boldsymbol{v}})) \\
        &\le \sup_{\theta \in \mathcal{A}_{k}} V_{N}(x, \theta, \boldsymbol{\mu}(\tilde{\boldsymbol{v}})) = V_{N}(z,k). 
    \end{split}
\end{equation*}

Note that
\begin{equation*}
\begin{split}
    V_{N}&(z,k) = V_{N}(x, \boldsymbol{\mu}(\tilde{\boldsymbol{v}}),k) 
    \ge V_{N}(x, \boldsymbol{\mu}(\boldsymbol{v}),k) \\
    &= l(x, \mu(x,v_{0})) \\
    &+ \sup_{\theta \in \mathcal{A}_{k}}\mathbb{E}_{w}[V_{N-1}(f(x,\mu(x,v_{0}),w,
    \theta),\theta,\boldsymbol{\mu}_{1:N-1}(\boldsymbol{v}))].
\end{split}
\end{equation*}

Thus we have
\begin{equation*}
    \begin{split}
        &\sup_{\theta \in \mathcal{A}_{k}}\mathbb{E}_{w}[V_{N}(z^{+},k+1)] - V_{N}(z,k) \\
        \le &\sup_{\theta \in \mathcal{A}_{k}}\mathbb{E}_{w}[V_{N}(z^{+},k)] - V_{N}(z,k) \\
        = &\sup_{\theta \in \mathcal{A}_{k}}\mathbb{E}_{w}[V_{N}(f(x,\mu(x,v_{0}),w,\theta),\theta,\boldsymbol{\mu}(\tilde{\boldsymbol{v}})^{+})] - V_{N}(z,k) \\
        \le &\sup_{\theta \in \mathcal{A}_{k}}\mathbb{E}_{w}[V_{N}(f(x,\mu(x,v_{0}),w,\theta),\theta,\boldsymbol{\mu}(\tilde{\boldsymbol{v}})^{+})] - l(x, \mu(x,v_{0})) \\
        &- \sup_{\theta \in \mathcal{A}_{k}}\mathbb{E}_{w}[V_{N-1}(f(x,\mu(x,v_{0}),w,
    \theta),\theta,\boldsymbol{\mu}_{1:N-1}(\boldsymbol{v}))] \\
        \le & - l(x, \mu(x,v_{0})) + \sup_{\theta \in \mathcal{A}_{k}}\mathbb{E}_{w}[V_{N}(f(x,\mu(x,v_{0}),w,\theta),\theta,\boldsymbol{\mu}(\tilde{\boldsymbol{v}})^{+}) \\
        &- V_{N-1}(f(x,\mu(x,v_{0}),w,
    \theta),\theta,\boldsymbol{\mu}_{1:N-1}(\boldsymbol{v}))].
    \end{split}
\end{equation*}
Since 
$
\boldsymbol{\mu}(\tilde{\boldsymbol{v}})^{+}=\{\boldsymbol{\mu}_{1:N-1}(\boldsymbol{v}), \mu(\cdot, v_{f})\},
$
we have
\begin{equation*}
    \begin{split}
        &V_{N}(x_{0},\theta,\boldsymbol{\mu}(\tilde{\boldsymbol{v}})^{+}) 
        - V_{N-1}(x_{0},\theta,\boldsymbol{\mu}_{1:N-1}(\boldsymbol{v})) \\
        =&l(x_{N-1},\mu(x_{N-1},v_{f}))-V_{f}(x_{N-1})\\
        &+\sup_{\theta \in \mathcal{A}_{k}}\mathbb{E}_{w}[V_{f}(f(x_{N-1},\mu(x_{N-1},v_{f}),w,\theta))].
    \end{split}
\end{equation*}
With the same definition of $\delta_{k}$ as in \eqref{eq:delta}, we can obtain 
\begin{equation*}
    \begin{split}
        \sup_{\theta \in \mathcal{A}_{k}}\mathbb{E}_{w}[V_{N}(z^{+},k+1)] - V_{N}(z,k) 
        \le  -l(x, \mu(x,v_{0})) + \delta_{k}.
    \end{split}
\end{equation*}

Then follow the same procedure as in Appendix \ref{appendix:Bayesian RA MPC stability}, we can prove that when $k\rightarrow \infty$,
\begin{equation*}
    \begin{split}
        \mathbb{E}_{w}[V_{N}(z^{+},k+1)] - V_{N}(z,k) 
        &\le -l(x, \mu(x,v_{0})) \\
        &\le -\alpha_{l}(|(x,\mu(x,v_{0}))|).
    \end{split}
\end{equation*}

According to Proposition \ref{prop: link pair}, we have
\begin{equation*}
    \begin{split}
        |(x,\boldsymbol{\mu}(\tilde{\boldsymbol{v}}))| \le |x|+|\boldsymbol{\mu}(\tilde{\boldsymbol{v}})| 
        &\le |x| + \alpha_{r}(|x|) := \alpha_{r'}(|x|) \\
        &\le \alpha_{r'}(|(x,\mu(x,v_{0}))|),
    \end{split}
\end{equation*}
Therefore, $\alpha_{l}\circ \alpha_{r'}^{-1}(|(x,\boldsymbol{\mu}(\tilde{\boldsymbol{v}}))|) \le \alpha_{l}(|(x,\mu(x,v_{0}))|)$.
Define $\alpha_{3}:=\alpha_{l}\circ \alpha_{r'}^{-1}$, we then obtain
\begin{equation*}
    \begin{split}
        \mathbb{E}_{w}[V_{N}(z^{+},k+1)] - V_{N}(z,k) 
        &\le -\alpha_{3}(|(x,\boldsymbol{\mu}(\tilde{\boldsymbol{v}}))|) \\
        &= -\alpha_{3}(|z|).
    \end{split}
\end{equation*}

Therefore, the time-varying value function $V_{N}(z,k)$ is a valid Lyapunov function for the difference inclusion system \eqref{eqn:extended state}, when $k$ is sufficiently large. Asymptotic stability follows directly from Theorem \ref{thm:Lyapunov Stability (Difference Inclusion)}
\end{proof}

Now we can derive a bound on state $x$ from the extended state $z$. 

\begin{proof}
Since $z$ is asymptotic stable, then according to the triangle inequality, we have
\begin{equation*}
    \begin{split}
        \mathbb{E}[|z_{i}|] &\le \beta(|z|,i) = \beta(|(x,\boldsymbol{\mu}(\tilde{\boldsymbol{v}}))|,i) \\
        &= \beta(|(x,\boldsymbol{0})+(0,\boldsymbol{\mu}(\tilde{\boldsymbol{v}}))|,i) \\ 
        &\le \beta(|(x,\boldsymbol{0})|+|(0,\boldsymbol{\mu}(\tilde{\boldsymbol{v}}))|,i) \\
        &\le \beta(|x|+|\boldsymbol{\mu}(\tilde{\boldsymbol{v}})|,i).
    \end{split}
\end{equation*}
From Proposition \ref{prop: link pair}, we then obtain
$$
\beta(|x|+|\boldsymbol{\mu}(\tilde{\boldsymbol{v}})|,i) \le \beta(|x|+\alpha_{r}(|x|),i) := \tilde{\beta}(|x|,i),
$$
with $\tilde{\beta} \in \mathcal{KL}$.

Therefore by utilizing the $l_{p}$-norm property that for all vectors $\boldsymbol{a}, \boldsymbol{b}$, $|(\boldsymbol{a}, \boldsymbol{b})| \ge |\boldsymbol{b}|$, we have
\begin{equation*}
    \begin{split}
        \mathbb{E}[|x_{i}|] \le \mathbb{E}[|(x_{i},\boldsymbol{\mu}(\tilde{\boldsymbol{v}}))|] = \mathbb{E}[|z_{i}|] 
        \le \tilde{\beta}(|x|,i),
    \end{split}
\end{equation*}
which implies that $\mathbb{E}[|x_{i}|] \le \tilde{\beta}(|x|,i)$. Thus we have also proved the bound on the evolution of state $x$ depending on only the $x$ initial condition, which implies the asymptotic stability of state $x$.
\end{proof}

\bibliographystyle{unsrt}
\bibliography{references}

@article{Wu2018AAsymptotics,
    title = {{A Bayesian risk approach to data-driven stochastic optimization: Formulations and asymptotics}},
    year = {2018},
    journal = {SIAM Journal on Optimization},
    author = {Wu, Di and Zhu, Helin and Zhou, Enlu},
    number = {2},
    pages = {1588--1612},
    volume = {28},
    doi = {10.1137/16M1101933},
    issn = {10526234},
    arxivId = {1609.08665}
}

@article{Kellett2014AResults,
    title = {{A compendium of comparison function results}},
    year = {2014},
    journal = {Mathematics of Control, Signals, and Systems},
    author = {Kellett, Christopher M.},
    number = {3},
    pages = {339--374},
    volume = {26},
    publisher = {Springer-Verlag London Ltd},
    doi = {10.1007/s00498-014-0128-8},
    issn = {1435568X}
}

@article{Singh2019AAlgorithms,
    title = {{A Framework for Time-Consistent, Risk-Sensitive Model Predictive Control: Theory and Algorithms}},
    year = {2019},
    journal = {IEEE Transactions on Automatic Control},
    author = {Singh, Sumeet and Chow, Yinlam and Majumdar, Anirudha and Pavone, Marco},
    number = {7},
    month = {7},
    pages = {2905--2912},
    volume = {64},
    publisher = {Institute of Electrical and Electronics Engineers Inc.},
    doi = {10.1109/TAC.2018.2874704},
    issn = {15582523},
    arxivId = {1703.01029}
}

@article{Schuurmans2023ASystems,
    title = {{A General Framework for Learning-Based Distributionally Robust MPC of Markov Jump Systems}},
    year = {2023},
    journal = {IEEE Transactions on Automatic Control},
    author = {Schuurmans, Mathijs and Patrinos, Panagiotis},
    number = {5},
    month = {5},
    pages = {2950--2965},
    volume = {68},
    publisher = {Institute of Electrical and Electronics Engineers Inc.},
    url = {https://ieeexplore.ieee.org/document/10019580/},
    doi = {10.1109/TAC.2023.3237999},
    issn = {0018-9286}
}

@article{Willems2005AExcitation,
    title = {{A note on persistency of excitation}},
    year = {2005},
    journal = {Systems and Control Letters},
    author = {Willems, Jan C. and Rapisarda, Paolo and Markovsky, Ivan and De Moor, Bart L.M.},
    number = {4},
    month = {4},
    pages = {325--329},
    volume = {54},
    doi = {10.1016/j.sysconle.2004.09.003},
    issn = {01676911}
}

@article{Kohler2020ASystems,
    title = {{A robust adaptive model predictive control framework for nonlinear uncertain systems}},
    year = {2020},
    journal = {International Journal of Robust and Nonlinear Control},
    author = {K{\"{o}}hler, Johannes and K{\"{o}}tting, Peter and Soloperto, Raffaele and Allg{\"{o}}wer, Frank and M{\"{u}}ller, Matthias A.},
    pages = {1--19},
    doi = {10.1002/rnc.5147},
    issn = {10991239},
    arxivId = {1911.02899}
}

@article{Smith2022AData,
    title = {{A step-by-step tutorial on active inference and its application to empirical data}},
    year = {2022},
    journal = {Journal of Mathematical Psychology},
    author = {Smith, Ryan and Friston, Karl J. and Whyte, Christopher J.},
    month = {4},
    volume = {107},
    publisher = {Academic Press Inc.},
    doi = {10.1016/j.jmp.2021.102632},
    issn = {10960880},
    pmid = {35340847}
}

@article{Doucet2010ALater,
    title = {{A tutorial on particle filtering and smoothing: fifteen years later}},
    year = {2010},
    journal = {Handbook of Nonlinear Filtering},
    author = {Doucet, Arnaud and Johansen, Adam M},
    number = {December},
    pages = {4--6},
    url = {https://www.seas.harvard.edu/courses/cs281/papers/doucet-johansen.pdf%0Ahttp://automatica.dei.unipd.it/tl_files/utenti/lucaschenato/Classes/PSC10_11/Tutorial_PF_doucet_johansen.pdf%5Cnpapers2://publication/uuid/6C713A28-1CAD-4A97-97D3-0168D610F0DF}
}

@book{Pedersen2012AnalysisNow,
    title = {{Analysis now}},
    year = {2012},
    booktitle = {Springer Science {\&} Business Media},
    author = {Pedersen, Gert K.},
    volume = {118},
    url = {https://books.google.com/books?hl=zh-CN&lr=&id=zNbiBwAAQBAJ&oi=fnd&pg=PA1&dq=Analysis+now&ots=N3W4k5DLIC&sig=ryHqIWATQgJNmPWwAwNTQ73TNak#v=onepage&q=Analysis now&f=false},
    doi = {10.5860/choice.26-5699},
    issn = {0009-4978}
}

@book{Vaart1998AsymptoticStatistics,
    title = {{Asymptotic Statistics}},
    year = {1998},
    author = {Vaart, A. W. van der},
    month = {10},
    publisher = {Cambridge University Press},
    isbn = {9780511802256},
    doi = {10.1017/CBO9780511802256}
}

@inproceedings{Wabersich2020BayesianSampling,
    title = {{Bayesian model predictive control: Efficient model exploration and regret bounds using posterior sampling}},
    year = {2020},
    booktitle = {Proceedings of the 2nd Conference on Learning for Dynamics and Control},
    author = {Wabersich, Kim P. and Zeilinger, Melanie N.},
    pages = {455--464},
    volume = {},
    issn = {23318422},
    arxivId = {2005.11744}
}

@inproceedings{Guzman2022BayesianUncertainty,
    title = {{Bayesian Optimisation for Robust Model Predictive Control under Model Parameter Uncertainty}},
    year = {2022},
    booktitle = {IEEE International Conference on Robotics and Automation},
    author = {Guzman, Rel and Oliveira, Rafael and Ramos, Fabio},
    isbn = {9781728196800}
}

@article{Liu2024BayesianData,
    title = {{Bayesian Stochastic Gradient Descent for Stochastic Optimization with Streaming Input Data}},
    year = {2024},
    journal = {SIAM Journal on Optimization},
    author = {Liu, Tianyi and Lin, Yifan and Zhou, Enlu},
    number = {1},
    month = {3},
    pages = {389--418},
    volume = {34},
    url = {https://epubs.siam.org/doi/10.1137/22M1478951},
    doi = {10.1137/22M1478951},
    issn = {1052-6234}
}

@article{Coppens2022Data-DrivenSystems,
    title = {{Data-Driven Distributionally Robust MPC for Constrained Stochastic Systems}},
    year = {2022},
    journal = {IEEE Control Systems Letters},
    author = {Coppens, Peter and Patrinos, Panagiotis},
    pages = {1274--1279},
    volume = {6},
    publisher = {IEEE},
    doi = {10.1109/LCSYS.2021.3091628},
    issn = {24751456},
    arxivId = {2103.03006}
}

@article{Berberich2021Data-DrivenGuarantees,
    title = {{Data-Driven Model Predictive Control with Stability and Robustness Guarantees}},
    year = {2021},
    journal = {IEEE Transactions on Automatic Control},
    author = {Berberich, Julian and Kohler, Johannes and Muller, Matthias A. and Allgower, Frank},
    number = {4},
    pages = {1702--1717},
    volume = {66},
    doi = {10.1109/TAC.2020.3000182},
    issn = {15582523},
    arxivId = {1906.04679}
}

@article{Coulson2019Data-enabledDeepc,
    title = {{Data-enabled predictive control: In the shallows of the deepc}},
    year = {2019},
    journal = {2019 18th European Control Conference, ECC 2019},
    author = {Coulson, Jeremy and Lygeros, John and Dorfler, Florian},
    pages = {307--312},
    publisher = {EUCA},
    isbn = {9783907144008},
    doi = {10.23919/ECC.2019.8795639},
    arxivId = {1811.05890}
}

@article{Coulson2021DistributionallyControl,
    title = {{Distributionally Robust Chance Constrained Data-enabled Predictive Control}},
    year = {2021},
    journal = {IEEE Transactions on Automatic Control},
    author = {Coulson, Jeremy and Lygeros, John and Dorfler, Florian},
    month = {8},
    pages = {1--1},
    publisher = {Institute of Electrical and Electronics Engineers (IEEE)},
    doi = {10.1109/tac.2021.3097706},
    issn = {0018-9286},
    arxivId = {2006.01702}
}

@article{Thangavel2018DualApproach,
    title = {{Dual robust nonlinear model predictive control: A multi-stage approach}},
    year = {2018},
    journal = {Journal of Process Control},
    author = {Thangavel, S. and Lucia, S. and Paulen, R. and Engell, S.},
    pages = {39--51},
    volume = {72},
    publisher = {Elsevier Ltd},
    url = {https://doi.org/10.1016/j.jprocont.2018.10.003},
    doi = {10.1016/j.jprocont.2018.10.003},
    issn = {09591524}
}

@inproceedings{Curi2020EfficientPlanning,
    title = {{Efficient Model-Based Reinforcement Learning through Optimistic Policy Search and Planning}},
    year = {2020},
    booktitle = {Advances in Neural Information Processing Systems},
    author = {Curi, Sebastian and Berkenkamp, Felix and Krause, Andreas},
    pages = {14156--14170}
}

@article{Guzman2021HeteroscedasticControl,
    title = {{Heteroscedastic bayesian optimisation for stochastic model predictive control}},
    year = {2021},
    journal = {IEEE Robotics and Automation Letters},
    author = {Guzman, Rel and Oliveira, Rafael and Ramos, Fabio},
    number = {1},
    pages = {56--63},
    volume = {6},
    issn = {23318422}
}

@article{Rosolia2018LearningFramework,
    title = {{Learning model predictive control for iterative tasks. A data-driven control framework}},
    year = {2018},
    journal = {IEEE Transactions on Automatic Control},
    author = {Rosolia, Ugo and Borrelli, Francesco},
    number = {7},
    pages = {1883--1896},
    volume = {63},
    publisher = {IEEE},
    doi = {10.1109/TAC.2017.2753460},
    issn = {00189286},
    arxivId = {1609.01387}
}

@inproceedings{Schuurmans2020Learning-BasedFeasibility,
    title = {{Learning-Based Distributionally Robust Model Predictive Control of Markovian Switching Systems with Guaranteed Stability and Recursive Feasibility}},
    year = {2020},
    booktitle = {Proceedings of the IEEE Conference on Decision and Control},
    author = {Schuurmans, Mathijs and Patrinos, Panagiotis},
    month = {12},
    pages = {4287--4292},
    volume = {2020-December},
    publisher = {Institute of Electrical and Electronics Engineers Inc.},
    isbn = {9781728174471},
    doi = {10.1109/CDC42340.2020.9303988},
    issn = {25762370},
    arxivId = {2009.04422}
}

@article{Koller2019Learning-BasedExploration,
    title = {{Learning-Based Model Predictive Control for Safe Exploration}},
    year = {2019},
    journal = {Proceedings of the IEEE Conference on Decision and Control},
    author = {Koller, Torsten and Berkenkamp, Felix and Turchetta, Matteo and Krause, Andreas},
    number = {Cdc},
    pages = {6059--6066},
    volume = {2018-Decem},
    publisher = {IEEE},
    isbn = {9781538613955},
    doi = {10.1109/CDC.2018.8619572},
    issn = {07431546},
    arxivId = {1803.08287}
}

@article{Hewing2020Learning-BasedControl,
    title = {{Learning-Based Model Predictive Control: Toward Safe Learning in Control}},
    year = {2020},
    journal = {Annual Review of Control, Robotics, and Autonomous Systems},
    author = {Hewing, Lukas and Wabersich, Kim P. and Menner, Marcel and Zeilinger, Melanie N.},
    number = {1},
    pages = {269--296},
    volume = {3},
    doi = {10.1146/annurev-control-090419-075625},
    issn = {2573-5144}
}

@book{Shapiro2009LecturesProgramming,
    title = {{Lectures on Stochastic Programming}},
    year = {2009},
    author = {Shapiro, Alexander and Dentcheva, Darinka and Ruszczy{\'{n}}ski, Andrzej},
    month = {1},
    pages = {},
    publisher = {Society for Industrial and Applied Mathematics},
    url = {http://epubs.siam.org/doi/book/10.1137/1.9780898718751},
    address = {Philadelphia, PA},
    isbn = {978-0-89871-687-0},
    doi = {10.1137/1.9780898718751}
}

@article{Bof2018LyapunovSystems,
    title = {{Lyapunov Theory for Discrete Time Systems}},
    year = {2018},
    journal = {arXiv preprint arXiv:1809.05289},
    author = {Bof, Nicoletta and Carli, Ruggero and Schenato, Luca},
    month = {9},
    url = {http://arxiv.org/abs/1809.05289},
    arxivId = {1809.05289}
}

@book{Rawlings2017ModelDesign,
    title = {{Model Predictive Control: Theory, Computation, and Design}},
    year = {2017},
    author = {Rawlings, James B. and Mayne, David Q.},
    isbn = {9780975937730},
    doi = {10.1155/2012/240898},
    issn = {16875249}
}

@article{Chatterjee2015OnTechniques,
    title = {{On stability and performance of stochastic predictive control techniques}},
    year = {2015},
    journal = {IEEE Transactions on Automatic Control},
    author = {Chatterjee, Debasish and Lygeros, John},
    number = {2},
    month = {2},
    pages = {509--514},
    volume = {60},
    publisher = {Institute of Electrical and Electronics Engineers Inc.},
    doi = {10.1109/TAC.2014.2335274},
    issn = {00189286}
}

@article{Diaconis1986OnEstimates,
    title = {{On the Consistency of Bayes Estimates}},
    year = {1986},
    journal = {The Annals of Statistics},
    author = {Diaconis, Persi and Freedman, David},
    number = {1},
    pages = {1--26},
    volume = {14}
}

@article{Allan2017OnMPC,
    title = {{On the inherent robustness of optimal and suboptimal nonlinear MPC}},
    year = {2017},
    journal = {Systems {\&} Control Letters},
    author = {Allan, Douglas A. and Bates, Cuyler N. and Risbeck, Michael J. and Rawlings, James B.},
    month = {8},
    pages = {68--78},
    volume = {106},
    publisher = {North-Holland},
    doi = {10.1016/J.SYSCONLE.2017.03.005},
    issn = {0167-6911}
}

@article{Ning2021OnlineSystems,
    title = {{Online learning based risk-averse stochastic MPC of constrained linear uncertain systems}},
    year = {2021},
    journal = {Automatica},
    author = {Ning, Chao and You, Fengqi},
    month = {3},
    volume = {125},
    publisher = {Elsevier Ltd},
    doi = {10.1016/j.automatica.2020.109402},
    issn = {00051098}
}

@article{Ruszczynski2006OptimizationFunctions,
    title = {{Optimization of convex risk functions}},
    year = {2006},
    journal = {Mathematics of Operations Research},
    author = {Ruszczy{\'{n}}ski, Andrzej and Shapiro, Alexander},
    number = {3},
    month = {8},
    pages = {433--452},
    volume = {31},
    publisher = {INFORMS},
    url = {https://pubsonline.informs.org/doi/abs/10.1287/moor.1050.0186},
    doi = {10.1287/moor.1050.0186},
    issn = {0364765X}
}

@article{Wabersich2020PerformanceGuarantees,
    title = {{Performance and safety of Bayesian model predictive control: Scalable model-based RL with guarantees}},
    year = {2020},
    journal = {arXiv},
    author = {Wabersich, Kim P. and Zeilinger, Melanie N.},
    pages = {1--11},
    issn = {23318422},
    arxivId = {2006.03483}
}

@book{Klenke2014ProbabilityCourse,
    title = {{Probability Theory: A Comprehensive Course}},
    year = {2014},
    booktitle = {Handbook of Mathematics for Engineers and Scientists},
    author = {Klenke, Achim},
    isbn = {978-1-4471-5361-0},
    doi = {10.1007/978-1-4471-5361-0}
}

@article{Wang2022Risk-averseControl,
    title = {{Risk-averse autonomous systems: A brief history and recent developments from the perspective of optimal control}},
    year = {2022},
    journal = {Artificial Intelligence},
    author = {Wang, Yuheng and Chapman, Margaret P.},
    month = {10},
    volume = {311},
    publisher = {Elsevier B.V.},
    doi = {10.1016/j.artint.2022.103743},
    issn = {00043702},
    arxivId = {2109.08947}
}

@article{Sopasakis2019Risk-averseControl,
    title = {{Risk-averse model predictive control}},
    year = {2019},
    journal = {Automatica},
    author = {Sopasakis, Pantelis and Herceg, Domagoj and Bemporad, Alberto and Patrinos, Panagiotis},
    month = {2},
    pages = {281--288},
    volume = {100},
    publisher = {Elsevier Ltd},
    doi = {10.1016/j.automatica.2018.11.022},
    issn = {00051098},
    arxivId = {1704.00342}
}

@inproceedings{Sopasakis2019Risk-averseControlb,
    title = {{Risk-averse risk-constrained optimal control}},
    year = {2019},
    booktitle = {2019 18th European Control Conference, ECC 2019},
    author = {Sopasakis, Pantelis and Schuurmans, Mathijs and Patrinos, Panagiotis},
    month = {6},
    pages = {375--380},
    publisher = {Institute of Electrical and Electronics Engineers Inc.},
    isbn = {9783907144008},
    doi = {10.23919/ECC.2019.8796021},
    arxivId = {1903.06749}
}

@inproceedings{Li2022Risk-AwareLearning,
    title = {{Risk-Aware Model Predictive Control Enabled by Bayesian Learning}},
    year = {2022},
    booktitle = {IEEE American Control Conference},
    author = {Li, Yingke and Lin, Yifan and Zhou, Enlu and Zhang, Fumin},
    isbn = {9550191028}
}

@article{Chapman2022Risk-AwareSystems,
    title = {{Risk-Aware Stability of Discrete-Time Systems}},
    year = {2022},
    author = {Chapman, Margaret P. and Kalogerias, Dionysios S.},
    month = {11},
    url = {http://arxiv.org/abs/2211.12416},
    arxivId = {2211.12416}
}

@article{Kishida2024Risk-AwareInvariance,
    title = {{Risk-Aware Stability, Ultimate Boundedness, and Positive Invariance}},
    year = {2024},
    journal = {IEEE Transactions on Automatic Control},
    author = {Kishida, Masako},
    number = {1},
    month = {1},
    pages = {681--688},
    volume = {69},
    publisher = {Institute of Electrical and Electronics Engineers Inc.},
    doi = {10.1109/TAC.2023.3301817},
    issn = {15582523}
}

@article{Thangavel2018RobustEstimation,
    title = {{Robust Dual Multi-stage NMPC using Guaranteed Parameter Estimation}},
    year = {2018},
    journal = {IFAC-PapersOnLine},
    author = {Thangavel, Sakthi and Aboelnour, Mohamed and Lucia, Sergio and Paulen, Radoslav and Engell, Sebastian},
    number = {20},
    pages = {72--77},
    volume = {51},
    publisher = {Elsevier B.V.},
    url = {https://doi.org/10.1016/j.ifacol.2018.10.177},
    doi = {10.1016/j.ifacol.2018.10.177},
    issn = {24058963}
}

@article{Lorenzen2019RobustUpdate,
    title = {{Robust MPC with recursive model update}},
    year = {2019},
    journal = {Automatica},
    author = {Lorenzen, Matthias and Cannon, Mark and Allg{\"{o}}wer, Frank},
    pages = {461--471},
    volume = {103},
    publisher = {Elsevier Ltd},
    url = {https://doi.org/10.1016/j.automatica.2019.02.023},
    doi = {10.1016/j.automatica.2019.02.023},
    issn = {00051098}
}

@inproceedings{Bonzanini2020SafeTrees,
    title = {{Safe Learning-based Model Predictive Control under State- And Input-dependent Uncertainty using Scenario Trees}},
    year = {2020},
    booktitle = {IEEE Conference on Decision and Control},
    author = {Bonzanini, Angelo D and Paulson, Joel A and Mesbah, Ali},
    pages = {2448--2454},
    isbn = {9781728174471},
    doi = {10.1109/CDC42340.2020.9304310},
    issn = {07431546}
}

@article{Mesbah2018StochasticControl,
    title = {{Stochastic model predictive control with active uncertainty learning: A Survey on dual control}},
    year = {2018},
    journal = {Annual Reviews in Control},
    author = {Mesbah, Ali},
    pages = {107--117},
    volume = {45},
    publisher = {Elsevier Ltd},
    url = {https://doi.org/10.1016/j.arcontrol.2017.11.001},
    doi = {10.1016/j.arcontrol.2017.11.001},
    issn = {13675788}
}

@article{Mesbah2016StochasticResearch,
    title = {{Stochastic model predictive control: An overview and perspectives for future research}},
    year = {2016},
    journal = {IEEE Control Systems},
    author = {Mesbah, Ali},
    number = {6},
    pages = {30--44},
    volume = {36},
    publisher = {IEEE},
    doi = {10.1109/MCS.2016.2602087},
    issn = {1066033X}
}

@book{Bertsekas1996StochasticCase,
    title = {{Stochastic Optimal Control: The Discrete-Time Case}},
    year = {1996},
    author = {Bertsekas, Dimitri and Shreve, Steven E},
    volume = {5},
    publisher = {Athena Scientific},
    isbn = {978-1-886529-03-8}
}

@book{Lehmann2006TheoryEstimation,
    title = {{Theory of point estimation}},
    year = {2006},
    author = {Lehmann, Erich L and Casella, George},
    publisher = {Springer Science {\textbackslash}{\&} Business Media}
}

@article{Shapiro2021TutorialProgramming,
    title = {{Tutorial on risk neutral, distributionally robust and risk averse multistage stochastic programming}},
    year = {2021},
    journal = {European Journal of Operational Research},
    author = {Shapiro, Alexander},
    number = {1},
    pages = {1--13},
    volume = {288},
    publisher = {Elsevier B.V.},
    doi = {10.1016/j.ejor.2020.03.065},
    issn = {03772217}
}

\end{document}